\documentclass[12pt,twoside]{article}  
\usepackage{amssymb,amsfonts,amsmath,amsbsy,theorem,amscd,dina42e}
\usepackage{graphicx}
\newcommand{\ol}[1]{\overline{#1}}

\numberwithin{equation}{section}
\newcommand{\C}{\ensuremath{\mathbb{C}}}
\newcommand{\R}{\ensuremath{\mathbb{R}}}

\newcommand{\Rd}{\ensuremath{\mathbb{R}^d}}
\newcommand{\Rn}{\ensuremath{\mathbb{R}^n}}

\newcommand{\N}{\ensuremath{\mathbb{N}}}

\newcommand{\M}{\ensuremath{\mathcal{M}}}
\newcommand{\HS}{\ensuremath{\mathcal{H}}}

\newcommand{\sd}{\, d}

\newcommand{\eps}{\ensuremath{\varepsilon}}
\newcommand{\weight}[1]{\langle #1\rangle}

\newcommand{\Div}{\operatorname{div}}

\newcommand{\E}{\mathcal{E}}

\newcommand{\ve}{\mathbf{v}}

\newcommand{\vc}[1]{\mathbf{#1}}
\newcommand{\tn}[1]{\mathbb{#1}}

\newcommand{\Chi}{\mathcal{X}}
\newcommand{\Linftyw}{\ensuremath{L^\infty_{\omega\ast}}}

\newcommand{\Ha}{\ensuremath{\mathcal{H}}}

\newcommand{\STC}{\sigma}
\newcommand{\RM}{\lambda}

\newtheorem{thm}{THEOREM}[section]

\newtheorem{lem}[thm]{Lemma}
\newtheorem{defn}[thm]{Definition}
\newtheorem{theorem}[thm]{Theorem}

\newtheorem{claim*}{Claim}
\newtheorem{assumption}[thm]{Assumption}
\theorembodyfont{\rmfamily}
\newtheorem{rem}[thm]{Remark}

\newenvironment{proof*}[1]{{\bf Proof
#1:}}{\hspace*{\fill}\rule{1.2ex}{1.2ex}\\ } 
\newenvironment{proof}{{\bf
Proof:\,}}{\hspace*{\fill}\rule{1.2ex}{1.2ex}\\ }

\newcommand{\no}{\mathbf{n}}
\newcommand{\Je}{\mathbf{J}}
\newcommand{\phie}{\boldsymbol{\varphi}}
\newcommand{\etae}{\boldsymbol{\eta}}
\newcommand{\nue}{\boldsymbol{\nu}}
\newcommand{\pe}{\boldsymbol{p}}

\begin{document}

\begin{titlepage}
\title{On Sharp Interface Limits for Diffuse Interface Models for Two-Phase Flows}
\author{Helmut Abels\footnote{Fakult\"at f\"ur Mathematik,  
Universit\"at Regensburg,
93040 Regensburg,
Germany, e-mail: {\sf helmut.abels@mathematik.uni-regensburg.de}}\ \ and Daniel Lengeler\footnote{Fakult\"at f\"ur Mathematik,  
Universit\"at Regensburg,
93040 Regensburg,
Germany, e-mail: {\sf daniel.lengeler@mathematik.uni-regensburg.de}}}
%\date{}
\end{titlepage}
\maketitle

\begin{abstract}
We discuss the sharp interface limit of a diffuse interface model for a two-phase flow of two partly miscible viscous Newtonian fluids of different densities, when a certain parameter $\eps>0$ related to the interface thickness tends to zero. In the case that the mobility stays positive or tends to zero slower than linearly in $\eps$ we will prove that weak solutions tend to varifold solutions of a corresponding sharp interface model. But, if the mobility tends to zero faster than $\eps^3$ we will show that certain radially symmetric solutions tend to functions, which will not satisfy the Young-Laplace law at the interface in the limit. 
\end{abstract}
{\small\noindent
{\bf Mathematics Subject Classification (2000):}
%Primary: 35R35; Secondary  35Q30, 76D27,  76D45, 76T99.\vspace{0.1in}\\
%76T99, %% Two-Phase flows: Others
%76D27, %% Incompressible viscous fluids: Other free-boundary flows; Hele-Shaw flows
%76D03, %% Incompressible viscous fluids: Existence, uniqueness, and regularity theory 
%76D05, %% Incompressible viscous fluids: Navier-Stokes equations
%76D45, %% Incompressible viscous fluids: Capillarity (surface tension)
Primary: 76T99; Secondary:
35Q30, %% Stokes and Navier-Stokes eq.
35Q35, %% Other equations arising in fluid mechanics
35R35,
76D05, %% Incompressible viscous fluids: Navier-Stokes equations
76D45\\ %% Incompressible viscous fluids: Capillarity (surface tension)
{\bf Key words:} Two-phase flow, diffuse interface model, sharp interface limit, Navier-Stokes system, Free boundary problems
%  Mullins-Sekerka equation,% convergence to equilibria. \vspace{0.1in}\\
}

\section{Introduction}

The present contribution is devoted to the study of the relations between so-called diffuse and sharp interface models for the flow of two  viscous incompressible Newtonian fluids. Such two-phase flows play a fundamental role in many fluid dynamical applications in physics, chemistry, biology, and the engineering sciences. There are two basic types of models namely the (classical) sharp interface models, where the interface $\Gamma(t)$ between the fluids is modeled as a (sufficiently smooth) surface and so-called diffuse interface models, where the ``sharp'' interface $\Gamma(t)$ is replaced by an interfacial region, where a suitable order parameter (e.g. the difference of volume fractions) varies smoothly, but with a large gradient between two distinguished values (e.g. $\pm 1$ for the difference of volume fractions). Then the natural question arises how diffuse and sharp interface models are related if a suitable parameter $\eps>0$, which is related to the width of the diffuse interface, tends to zero. There are several results on this question, which are based on formally matched asymptotics calculations. But so far there are very few mathematically rigorous convergence results.

More precisely, we study throughout the paper the sharp interface limit of the following diffuse interface model:
\begin{alignat}{2}\label{eq:NSCH1}
  \rho\partial_t \ve + \left(\rho\ve+ \frac{\partial \rho}{\partial c} \vc{J}\right)\cdot \nabla \ve - \Div (\nu (c)D\ve) + \nabla p &= -\eps\Div
  (a(c)\nabla c \otimes \nabla c) 
  &\ &\text{in}\ Q, \\\label{eq:NSCH2}
  \Div \ve &=0 &&\text{in}\ Q, \\\label{eq:NSCH3}
  \partial_t c + \ve\cdot\nabla c &= \Div (m_\eps(c)\nabla \mu) & &\text{in}\ Q, \\\label{eq:NSCH4}
\mu &= \eps^{-1}f'(c) - \eps\Delta c& & \text{in}\ Q,\\\label{eq:NSCH5}
  \ve|_{\partial\Omega} &=0 & & \text{on}\ S,\\\label{eq:NSCH6}
  \no_{\partial\Omega}\cdot \nabla  c|_{\partial\Omega} = \no_{\partial\Omega}\cdot  \nabla \mu|_{\partial\Omega} &= 0
  & & \text{on}\ S,\\\label{eq:NSCH7}
  (\ve,c)|_{t=0} &= (\ve_0,c_0) && \text{in}\ \Omega,
\end{alignat}
where $Q= \Omega\times (0,\infty), S= \partial\Omega\times (0,\infty)$, $\Omega\subseteq \Rn$ is a suitable domain, and
$
  \vc{J}= - m_\eps(c) \nabla \mu.
$
Here $c=c_2-c_1$ is the volume fraction difference of the fluids, $\rho=\rho(c)$ is the density of the fluid mixture, depending explicitly on $c$ through $\rho(c)=\tfrac{\tilde{\rho}_2-\tilde{\rho}_1}2 c-\tfrac{\tilde{\rho}_1+\tilde{\rho}_2}2$, where $\tilde{\rho}_j$ is the specific density of fluid $j=1,2$, and $f$ is a suitable
``double-well potential'' e.g. $f(c)=\frac18(1-c^2)^2$. Precise assumptions will be made below. Moreover,
$\eps>0$ is a small parameter related to the interface thickness, $\mu$
is the so-called chemical potential, $m_\eps(c)>0$ a mobility coefficient related to the strength of diffusion in the mixture and $a(c)$ is a coefficient in front of the $|\nabla c|^2$-term in the free energy of the system. Finally, $\no_{\partial\Omega}$ denotes the exterior normal of $\partial\Omega$. 
The model was derived by A., Garcke, and Gr\"un \cite{AbelsGarckeGruen2}. In the case $\rho(c)\equiv const.$ it coincides with the so-called ``Model H'' in Hohenberg and Halperin~\cite{HohenbergHalperin}, cf. also
Gurtin et al.~\cite{GurtinTwoPhase}. Existence of weak solutions for this system in the case of a bounded, sufficiently smooth domain $\Omega$ and for a suitable class of singular free energy densities $f$ was proved by A., Depner, and Garcke~\cite{AbelsDepnerGarcke}. We refer to the latter article for further references concerning analytic results for this diffuse interface model in the case $\rho(c)\equiv const.$ and related models.

In \cite{AbelsGarckeGruen2} the sharp interface limit $\eps\to 0$ was discussed with the method of formally matched asymptotics. It was shown that for the scaling $m_\eps(c)\equiv \tilde{m}\eps^\alpha$ with $\alpha =0,1$, $\tilde{m}>0$, solutions of the system \eqref{eq:NSCH1}-\eqref{eq:NSCH5} converges to solutions of
\begin{alignat}{2}\label{eq:1}
  \rho^\pm\partial_t \ve + (\rho^\pm \ve + \tfrac{\tilde{\rho}_2-\tilde{\rho}_1}2\vc{J})\cdot \nabla \ve  - \Div \tn{T}^\pm(\ve,p) &= 0 &\ & \text{in}\
  \Omega^\pm(t), t>0, \\\label{eq:2}
\Div \ve &= 0 &\ & \text{in}\ \Omega^\pm(t), t>0, \\\label{eq:2'}
m_0 \Delta\mu &= 0 &\ & \text{in}\ \Omega^\pm(t), t>0, \\\label{eq:3}
-\no\cdot [\tn{T}(\ve,p)] &= \STC H \no  &\ & \text{on}\ \Gamma(t), t>0,
  \\\label{eq:4}
V - \no\cdot \ve|_{\Gamma(t)}&=  - [\tfrac{m_0}2\no\cdot \nabla\mu] && \text{on}\ \Gamma(t),t>0,\\\label{eq:4'}
\mu|_{\Gamma(t)} & = \STC H&& \text{on}\ \Gamma(t),t>0,
\end{alignat}
with $\vc{J}= -m_0 \nabla \mu$. Here $\no$ denotes the unit
normal of $\Gamma(t)$ that points inside $\Omega^+(t)$ and $V$ and
$H$ the normal velocity and scalar mean curvature of $\Gamma(t)$
with respect to $\no$. Moreover, by $[\cdot]$ we denote
the jump of a quantity across the interface in direction of $\no$, i.e.,
$[f](x)= \lim_{h\to 0}(f(x+h\no)-f(x-h\no))$ for $x\in \Gamma(t)$. Furthermore, $\STC$ is
a surface tension coefficient determined uniquely by $f$ and $m_0=\tilde{m}$ if $\alpha=0$ and $m_0=0$ if $\alpha=1$ is a mobility constant. Implicitly it is
assumed that $\ve,\mu$ do not jump across $\Gamma(t)$,
i.e., $$[\ve]=[\mu]=0\qquad \text{on}\ \Gamma(t), t>0.$$ 
In the following we close the system with the boundary and initial conditions
\begin{alignat}{2}\label{eq:5}
\ve|_{\partial\Omega} &= 0 &\ & \text{on}\ \partial\Omega,t>0, \\\label{eq:5'}
 \no_{\partial\Omega}\cdot m_\eps(c)\nabla  \mu|_{\partial\Omega}  &= 0 &\ & \text{on}\ \partial\Omega,t>0,\\
\Omega^+(0)&= \Omega_0^+,&& \label{eq:6'}\\\label{eq:6}
\ve|_{t=0} &= \ve_0 &\ & \text{in}\ \Omega,
\end{alignat}
where $\ve_0, \Omega_0^+$ are given initial data satisfying
$\partial\Omega_0^+\cap \partial\Omega=\emptyset$.  
Equations (\ref{eq:1})-(\ref{eq:2}) describe the conservation of linear
momentum and mass in both fluids, and (\ref{eq:3}) is the balance of forces
at the boundary. %(\ref{eq:4}) is the kinematic condition that the
%interface is given by the flow of the mass particles and the jump in the
%gradient of the chemical potential. 
The equations for $\ve$ are
complemented by the non-slip condition \eqref{eq:5} at the boundary of
$\Omega$. 
The conditions \eqref{eq:2'}, \eqref{eq:5'} describe together with
\eqref{eq:4} a continuity equation for the mass of the phases, and
\eqref{eq:4'} relates the chemical potential $\mu$ to the $L^2$-gradient
of the surface area, which is given by the mean curvature of the
interface.

We note that in the case $\alpha=1$, i.e., $m_0=0$, (\ref{eq:4}) describes the usual kinematic condition that the 
interface is transported by the flow of the surrounding fluids and
(\ref{eq:1})-(\ref{eq:6}) reduces to the classical model of a two-phase
Navier--Stokes flow. Existence of strong solutions locally in time was first proved by Denisova and
Solonnikov~\cite{DenisovaTwoPhase}. We refer to Pr\"uss and Simonett~\cite{PruessSimonettTwoPhaseNSt} and K\"ohne et al.~\cite{KoehnePruessWilkeTwoPhase} for more recent results and further references. Existence of generalized solutions globally in times was shown by Plotnikov~\cite{PlotnikovTwoPhase} and
A.~\cite{GeneralTwoPhaseFlow,ReviewGeneralTwoPhaseFlow}. 
On the other hand, if $\alpha=0$, $m_0>0$, respectively,  the equations (\ref{eq:2'}), (\ref{eq:4'}), (\ref{eq:5'})
are a variant of the Mullins--Sekerka flow of a family of interfaces with an additional convection term $\no\cdot \ve|_{\Gamma(t)}$. In the case $\tilde{\rho}_1=\tilde{\rho}_2$ existence of weak solutions for large times and general initial data was proved by A. and R\"oger~\cite{NSMS} and existence of strong solutions locally in time and stability of spherical droplets was proved by A. and Wilke~\cite{StrongNSMS}.

In the following we address the following question: Under which assumptions on the behavior of $m_\eps(c)$ as $\eps\to 0$ do weak solutions of \eqref{eq:NSCH1}-\eqref{eq:NSCH7} converge to weak/generalized solutions of \eqref{eq:1}-\eqref{eq:6}? In this paper we provide a partial answer to that question. If one assumes e.g. $m_\eps(c)= \tilde{m}\eps^\alpha$, the results in the following will show that convergence holds true in the case $\alpha\in [0,1)$. More precisely, we will show that weak solutions of \eqref{eq:NSCH1}-\eqref{eq:NSCH7} converge to so-called varifold solutions of \eqref{eq:1}-\eqref{eq:6}, which are defined in the spirit of Chen~\cite{ChenLimitCH}.  But in the case $\alpha\in (3,\infty)$ we will construct radially symmetric solutions of \eqref{eq:NSCH1}-\eqref{eq:NSCH4} in the domain $\Omega=\{x\in\R: 1<|x|<M\}$ with suitable inflow and outflow boundary conditions, which do not converge to a solution of \eqref{eq:1}-\eqref{eq:4'}.  In particular, the pressure $p$ in the limit $\eps\to 0$ satisfies 
\begin{equation*}
   [p]=\sigma\kappa(t) H\qquad \text{on}\ \Gamma(t)= \partial B_{R(t)}(0),
\end{equation*}
where $R(t),\kappa(t)\to_{t\to\infty}\infty$ and $\ve$ is independent of $t$ and smooth in $\Omega$. This shows that the Young-Laplace law \eqref{eq:3} is not satisfied.  We note that these results are consistent with the numerical studies of Jacqmin, where a scaling of the mobility as $m_\eps(c)=\tilde{m}\eps^\alpha$ with $\alpha\in [1,2)$ was proposed and considered.

The structure of the article is as follows: First we introduce some notation and preliminary results in Section~\ref{sec:Prelim}. Then we prove our main result on convergence of weak solutions of \eqref{eq:NSCH1}-\eqref{eq:NSCH7} to varifold solutions of \eqref{eq:1}-\eqref{eq:6} in the case that the mobility $m_\eps(c)$ tends to zero as $\eps\to 0$ slower than linearly in Section~\ref{sec:SharpInterfaceLimit}. Finally, in Section~\ref{sec:NonConvergence}, we consider certain radially symmetric solutions of  \eqref{eq:NSCH1}-\eqref{eq:NSCH7} and show that these do not converge to a solution of \eqref{eq:1}-\eqref{eq:4'} if the mobility tends to zero too fast as $\eps\to 0$.

\section{Notation and Preliminaries}\label{sec:Prelim}

Let $U\subseteq \R^d$ be open.
 Then $\M(U;\R^N)$, $N\geq 1$, denotes the space of all finite $\R^N$-valued Radon measures on $U$. 
By the Riesz representation theorem $\M(U;\R^N)=C_0(U;\R^N)'$, cf. e.g.
\cite[Theorem 1.54]{AmbrosioEtAl}. Moreover, $\M(U):= \M(U,\R)$. Given $\lambda\in \M(U;\R^N)$ we denote by
$|\lambda|$ the total variation measure defined by 
\begin{equation*}
  |\lambda|(A) = \sup \left\{\sum_{k=0}^\infty |\lambda(A_k)|: A_k\in \mathcal{B}(U)\ \text{pairwise disjoint}, A= \bigcup_{k=0}^\infty A_k \right\}
\end{equation*}
for every $A\in \mathcal{B}(U)$, where $\mathcal{B}(U)$ denotes the
$\sigma$-algebra of Borel sets of $U$.  Moreover,
$\frac{\lambda}{|\lambda|}\colon U\to \R^N$
denotes the Radon-Nikodym derivative of $\lambda$ with respect to $|\lambda|$.
The restriction of a measure $\mu$
to a $\mu$-measurable set $A$ is denoted by $(\mu\lfloor A)(B)= \mu(A\cap B)$.
Furthermore, the $s$-dimensional Hausdorff measure on $\R^d$, $0\leq s\leq d$, is
denoted by $\Ha^s$.  
Recall that 
\begin{eqnarray*}
  BV(U)&=& \{f \in L^1(U): \nabla f \in \M(U;\R^d)\}\\
  \|f\|_{BV(U)}&=& \|f\|_{L^1(U)}+ \|\nabla f\|_{\M(U;\R^d)},
\end{eqnarray*}
where $\nabla f$ denotes the distributional derivative. Moreover,
$BV(U;\{0,1\})$ denotes the set of all $\Chi\in BV(U)$ such that
$\Chi(x)\in\{0,1\}$ for almost all $x\in U$.

% Note that by the Riesz representation theorem, if $f\in L^1(U)$, then $\nabla f\in \M(U;\R^d)$ if and only if
% \begin{equation*}
%   \sup \left\{ \left|\int_U f(x) \Div g(x) \sd x\right|: g\in C_0^\infty(U;\R^d), |g(x)|\leq 1\right\} <\infty
% \end{equation*}
% where the supremum above coincides with $\|\nabla f\|_{\M(U;\R^d)}$. 

A set $E\subseteq U$ is said to have finite perimeter in $U$ if
$\Chi_E\in BV(U)$. By the structure theorem of sets of finite perimeter
$|\nabla \Chi_E|=\Ha^{d-1}\lfloor \partial^\ast E$, where
$\partial^\ast E$ is the so-called reduced boundary of $E$ and for all
$\boldsymbol{\varphi}\in C_0(U,\Rd)$ 
\begin{equation*}
 -\weight{\nabla \Chi_E,\boldsymbol{\varphi}} = \int_E \Div \boldsymbol{\varphi} \sd x =
 -\int_{\partial^\ast E} \boldsymbol{\varphi}\cdot \no_E \sd \Ha^{d-1}, 
\end{equation*}
where $\no_E(x) = \frac{\nabla \Chi_E}{|\nabla \Chi_E|}$,
cf. e.g. \cite{AmbrosioEtAl}. Note that, if $E$ is a domain with
$C^1$-boundary, then $\partial^\ast E = \partial E$ and $\no_E$ coincides
with the interior unit normal.  

As usual the space of smooth and compactly supported functions in an open set
$U$ is denoted by $C_0^\infty(U)$. Moreover, $C^\infty(\ol{U})$ denotes the
set of all smooth functions $f\colon U\to \C$ such that all derivatives have continuous extensions on $\ol{U}$.   
 For $0<T\leq \infty$, we denote by
$L^p_{loc}([0,T);X)$, $1\leq p\leq \infty$,
the space of all strongly measurable $f\colon (0,T)\to X $ such that $f\in
L^p(0,T';X)$ for all $0<T'<T$. Here $L^p(M)$ and $L^p(M;X)$ denote the standard Lebesgue spaces for scalar and $X$-valued functions, respectively. 
Furthermore, 
$C_{0,\sigma}^\infty(\Omega)= \{\boldsymbol{\varphi}\in C_0^\infty(\Omega)^d:
\Div \boldsymbol{\varphi} =0\}$ and
\begin{equation*}
  L^2_\sigma(\Omega)\,:=\, \overline{C_{0,\sigma}^\infty(\Omega)}^{L^2(\Omega)}.
\end{equation*}
If $Y=X'$ is a dual space and  $Q\subseteq\R^N$ is open, then
$\Linftyw(Q;Y)$ denotes the space of all functions $\nu\colon Q\to Y$
that are weakly-$\ast$ measurable and essentially bounded, i.e.,
\begin{equation*}
  x\mapsto %\weight{\nu_x, F(x,.)}=
  \weight{\nu_x,F(x,.)}_{X',X}
\end{equation*}
is measurable for each $F\in L^1(Q;X)$ and
\begin{equation*}
  \|\nu\|_{\Linftyw(Q;Y)}:= \text{ess sup}_{x\in Q}\|\nu_x\|_{Y}<\infty.
\end{equation*}
Moreover, we note that there is  a separable Banach space $X$ such that
$X'=BV(\Omega)$, cf. \cite{AmbrosioEtAl}. %, which follows from the Hahn-Banach theorem and the fact
%that $BV(\Omega)$  can be identified with a closed subspace of
%$\M(\Omega)^{d+1}$ via the mapping $f\mapsto (f,\nabla f)$. 
As a
consequence \cite{Edwards} we obtain that $\Linftyw(0,T; BV(\Omega))=
\big(L^1(0,T;X)\big)^*$ and that uniformly bounded sets in $
\Linftyw(0,T; BV(\Omega))$ are weakly*-precompact.

%%% Local Variables:
%%% mode: latex
%%% TeX-master: "ChenSharpInterfaceLimit.tex"
%%% End:

\section{Sharp Interface Limit}\label{sec:SharpInterfaceLimit}

In this section we discuss the relation between (\ref{eq:1})-(\ref{eq:6}) and its diffuse
interface analogue (\ref{eq:NSCH1})-(\ref{eq:NSCH7}).

\begin{assumption}\label{assumptions}
We assume that the domain $\Omega\subset\R^d$, $d=2,3$ is bounded and smooth. Furthermore, we assume that there exist constants $c_0,C_0>0$ such that
\begin{itemize}
 \item  $f\in C^3(\R)$, $f(c)\geq 0$, $f(c)=0$ if and only if $c=-1,1$, and $f''(c)\geq c_0 |c|^{p-2}$ if $|c|\geq 1-c_0$ for some constant $p\ge 3$
\item $\rho,a,\nu\in C^1(\R)$ with $c_0\le \rho,a,\nu\le C_0$ and \[\rho(c)=\frac{\tilde\rho_2+\tilde\rho_1}{2} + \frac{\tilde\rho_2-\tilde\rho_1}{2}c\] for $c\in [-1,1]$
\item  $m_\eps,m_0\in C^1(\R)$, $0\le m_\eps,m_0\le C_0$, $m_\eps\rightarrow_{\eps\rightarrow 0} m_0$ in $C^1(\R)$, and either $m_0\ge c_0$ or $m_0\equiv 0$. If $m_0\equiv 0$, then $m_\eps\ge \overline{m}_\eps$ for constants $\overline{m}_\eps>0$ with $\overline{m}_\eps\rightarrow_{\eps\rightarrow 0} 0$.
\end{itemize}
\end{assumption}

The stronger assumption $p\ge 3$ (compared to $p>2$ in \cite{ChenLimitCH}) is needed here for the uniform estimate of $\ve_\eps \cdot \nabla c_\eps = \Div (\ve_\eps c_\eps)$ in $L^2(0,T;H^{-1}(\Omega))$. A possible choice for the homogeneous free energy density is $f(s) = (s^2-1)^2$. Moreover, let $\STC= \int_{-1}^1\sqrt{f(s)/2}\, ds$ and $A(s)=\int_0^s\sqrt{a(\tau)}\, d\tau$. 

Now, let us consider the energy identities corresponding to our two systems. We recall that every sufficiently smooth solution of the Navier--Stokes/Mullins--Sekerka system (\ref{eq:1})-(\ref{eq:6}) satisfies 
\begin{equation}\label{eq:en_id_nsms'}
  \frac{d}{dt} \frac12 \int_\Omega \rho(c)\,|\ve|^2\sd x +\STC \frac{d}{dt}
  \Ha^{d-1}(\Gamma) = - \int_\Omega \nu(c) |D\ve|^2 \sd x - \int_\Omega m_0(c)|\nabla \mu|^2 \sd x,
\end{equation}
where $c(t,x)=-1+2\chi_{\Omega^+(t)}(x)$. On the other hand, every sufficiently smooth solution of (\ref{eq:NSCH1})-(\ref{eq:NSCH7}) satisfies
\begin{equation}\label{eq:en_id_nsch}
  \frac{d}{dt} \frac12 \int_\Omega \rho(c)\,|\ve|^2\sd x +\frac{d}{dt} \E_\eps(c) = - \int_\Omega \nu(c) |D\ve|^2 \sd x -\int_\Omega m_\eps(c)|\nabla \mu|^2 \sd x, 
\end{equation}
where 
\begin{equation*}
  \E_\eps(c) = \int_\Omega \left(\eps\frac{|\nabla A(c)|^2}{2} + \frac{f(c)}{\eps}\right) \sd x
\end{equation*}
is the free energy. Moreover, by Modica and Mortola~\cite{ModicaMortola1} or Modica~\cite{Modica1}, for $A(c)=c$, we have
\begin{equation*}
  \E_\eps \to_{\eps\to 0} \mathcal{P} \qquad \text{w.r.t.}\ L^1\text{-}\Gamma\text{-convergence},
\end{equation*}
where
\begin{equation*}
  \mathcal{P}(u)=
  \begin{cases}
    \STC\, \Ha^{d-1}(\partial^\ast E) & \text{if}\ u = -1 + 2\chi_E \ \text{and}\ E\ \text{has finite perimeter},\\
+\infty& \text{else}.
  \end{cases}
\end{equation*}
Here, $\partial^\ast E$ denotes the reduced boundary. Note that $\partial^\ast E= \partial E$ if $E$ is a sufficiently regular domain. %%% Tonegawa: Convergence von 1. Variation??
Therefore, we see that the energy identity
\eqref{eq:en_id_nsms'} is formally identical to the sharp interface
limit of the energy identity \eqref{eq:en_id_nsch} of the diffuse interface
model (\ref{eq:NSCH1})-(\ref{eq:NSCH7}).
% In contrast, if we would choose $m=m_\eps\to_{\eps\to 0} 0$ in
% (\ref{eq:NSCH1})-(\ref{eq:NSCH7}), we formally obtain in
% the sharp interface limit the energy
% identity of the classical two-phase flow (\ref{eq:1})-(\ref{eq:6}). 
 %But a rigorous analysis of these different asymptotic limits remains to be done.

We will now adapt the arguments of Chen~\cite{ChenLimitCH}, see also A. and R\"oger~\cite{NSMS}, to show that, as
$\eps\to 0$, solutions of the diffuse interface model (\ref{eq:NSCH1})-(\ref{eq:NSCH7})
converge to \emph{varifold solutions} of the system (\ref{eq:1})-(\ref{eq:6}). Let $Q=\Omega\times (0,\infty)$ and $G_{d-1}:=S^{d-1}/\sim$ where $\nue_0\sim\nue_1$ for $\nue_0,\nue_1\in S^{d-1}$ iff $\nue_0=\pm\nue_1$ and $S^{d-1}$ is the unit sphere in $\R^d$.
\begin{defn}\label{defn:VarifoldSolution}
 Let $\ve_0\in L^2_\sigma(\Omega)$ and $E_0\subset \Omega$ be a set of finite
 perimeter. Then $(\ve,E,\mu,V)$ if $m_0>0$ and $(\ve,E,V)$ else is called a \emph{varifold solution} of
  (\ref{eq:1})-(\ref{eq:6}) with initial values $(\ve_0,E_0)$ if the following conditions are satisfied:
  \begin{enumerate}
  \item $\ve\in L^2((0,\infty);H^1(\Omega)^d)\cap L^\infty((0,\infty);L^2_\sigma(\Omega))$; $\mu\in L^2_{loc}([0,\infty);H^1(\Omega)), \nabla\mu\in L^2((0,\infty);L^2(\Omega)^d)$ if $m_0>0$.
\item $E=\bigcup_{t\geq 0} E_t\times \{t\}$ is a measurable subset of
  $\Omega\times [0,\infty)$ such that $\chi_E \in
  C([0,\infty);L^1(\Omega))\cap L^\infty_{\omega^\ast}((0,\infty);BV(\Omega))$ and $|E_t| = |E_0|$ for all $t\geq 0$.
%, and $\chi_{E_{t=0}}=\chi_{E_0}$ in $L^1(\Omega)$. 
\item $V$ is a Radon measure on $\overline\Omega\times G_{d-1}\times (0,\infty)$ %%% Definition P
  such that $V= V^t dt$ where $V^t$ is a Radon measure on $\overline\Omega \times G_{d-1}$ for almost all $t\in (0,\infty)$, i.e., a general varifold in $\overline\Omega$. Moreover, for almost all $t\in (0,\infty)$ $V^t$ has the representation
  \begin{equation}\label{eqn:V}
    \int_{\overline\Omega \times G_{d-1}} \psi (x,\pe) \sd V^t(x,\pe)  = \sum_{i=1}^d \int_{\overline\Omega}
    b_i^t(x)\, \psi(x,\pe_i^t(x))\sd \lambda^t(x) 
  \end{equation}
  for all $\psi\in
    C(\overline\Omega\times G_{d-1})$. Here, for almost all $t\in (0,\infty)$, $\lambda^t$ is a Radon measure on $\overline{\Omega}$, and the $\lambda^t$-measurable functions $b_i^t,\pe_i^t$ are $\R$- and $G_{d-1}$-valued, respectively, such that %%% Notation: \Omega oder \ol{\Omega}??
\[0\leq b_i^t \leq 1,\quad \sum_{i=1}^d b_i^t \geq 1,\quad \sum_{i=1}^d
     \pe_i^t\otimes \pe_i^t = I\qquad \lambda^t\text{-a.e.}\]
and
\[\frac{|\nabla\chi_{E_t}|}{\lambda^t} \leq \frac1{2\STC}.\]
%   \begin{alignat*}{1}
%     &0\leq b_i^t \leq 1,\quad \sum_{i=1}^d b_i^t \geq 1,\quad \sum_{i=1}^d
%     p_i^t\otimes p_i^t = I \quad \lambda^t\text{-almost everywhere}\\
%     &\frac{|\nabla \Chi_{E_t}|}{\lambda^t} \leq \frac1{\STC}\qquad
%     \lambda^t\text{-a.e. with}\ \STC= \int_{0}^1 \sqrt{2f(s)} \sd s.
%   \end{alignat*}
\item For $c:=-1+2\chi_E$, $\Je:=-m_0(c)\nabla\mu$ if $m_0>0$ and $\Je:=0$ else as well as $\tilde \Je:=\frac{\partial\rho}{\partial c}(c)\Je$ we have
\begin{align}\nonumber
  \int_{Q}\Big(-\rho(c) \ve\cdot\partial_t \phie  -  \ve\otimes \big(\rho(c)\ve + \tilde \Je\big): \nabla \phie +\nu(c) D\ve:D\phie\Big)  \,d(x,t)\\
 -\int_\Omega \rho(c|_{t=0})\,\ve_0\cdot\phie|_{t=0} \,dx 
   = - \int_0^\infty \weight{\delta V^t,\phie} \sd t \label{eq:ns'}
\end{align}
for all $\phie \in C^\infty_0([0,\infty); C_{0,\sigma}^\infty(\Omega))$ and
\begin{align}
  &2\int_{E}\partial_t\psi + \Div(\psi \ve) \,d(x,t) + \int_Q\Je\cdot\nabla\psi\, d(x,t)+\int_{E_0} \psi|_{t=0}\,dx=0 \label{eq:mu'} 
\end{align}
for all $\psi\in
C^\infty_0([0,\infty)\times \overline{\Omega})$. Here
\begin{equation*}{}
  \weight{\delta V^t,\phie}:= \int_{\ol{\Omega}\times G_{d-1}} ({I}-\vc{p}\otimes \vc{p}): \nabla \phie \sd (x,\vc{p})\quad \text{for all}\ \phie \in C^\infty(\overline{\Omega};\R^d).
\end{equation*}
Furthermore, if $m_0>0$ we have
\begin{equation}\label{eq:fvar}
  2\int_{E_t} \Div (\mu\etae)\, dx =  \left\langle \delta
    V^t,\etae\right\rangle 
%=\int_{\ol\Omega\times G_{d-1}} \nabla \etae (x):(I-p\otimes p) \sd V^t(x,p) 
\end{equation}
for all $\etae \in C^1_0(\Omega;\R^d)$ and almost all $t\in (0,\infty)$.
\item Finally, for almost all $0< s < t <\infty$
  \begin{align}\nonumber
    \frac12\int_\Omega\rho(c(t))|\ve(t)|^2\,dx + \lambda^t (\ol{\Omega})
+ \int_s^t\int_\Omega \left(\nu(c) |D\ve|^2 - \Je\cdot\nabla\mu\right)\, d(x,\tau)\\\label{eq:EnergyEstim} \leq \frac12\int_\Omega\rho(c(s))|\ve(s)|^2\,dx+\lambda^s (\ol{\Omega}).
  \end{align}
\end{enumerate}
\end{defn}
We define the free energy density by
\begin{equation*}
  e_\eps (c) = \eps \frac{|\nabla A(c)|^2}{2} + \frac{f(c)}{\eps}.
\end{equation*}
\smallskip
In \cite{ModelH} the existence of global weak solutions is shown for a class of singular free energies. We note that this proof can be easily carried over to the present situation with only minor modifications and even some simplifications since $f$ is non-singular. Throughout this paper we will use the definition of weak solutions in \cite{ModelH}. By this definition we have 
\begin{align*}
&\ve_\eps\in BC_\omega([0,\infty);L^2_\sigma(\Omega))\cap L^2((0,\infty);H^1(\Omega)^d),\\
&c_\eps\in BC_\omega([0,\infty);H^1(\Omega))\cap L^2_{loc}([0,\infty);H^2(\Omega)), f(c_\eps)\in L^2_{loc}([0,\infty);L^2(\Omega)),\\
&\mu\in L^2_{loc}([0,\infty);L^2(\Omega)), \nabla \mu\in L^2([0,\infty);L^2(\Omega)^d),
\end{align*}
and 
  \begin{align}\nonumber\label{eq:weakNSCH1}
    \int_Q -\rho(c_\eps)\ve_\eps\cdot\partial_t \phie -  \ve_\eps\otimes\big(\rho(c_\eps)\ve_\eps+\tilde\Je_\eps\big): \nabla \phie + \nu(c_\eps)D\ve_\eps: D \phie\, d(x,t)\\ 
- \int_\Omega \rho(c_{0,\eps})\, \ve_{0,\eps}\cdot\phie|_{t=0}\sd x  = 
%\int_Q\mu_\eps\nabla c_\eps\cdot\phie\, d(x,t)
\int_Q \eps\,a(c_\eps) \nabla c_\eps\otimes \nabla c_\eps : \nabla \phie\sd (x,t)
  \end{align}
 for $\Je_\eps:=-m_\eps(c_\eps)\nabla\mu_\eps$, $\tilde \Je_\eps:=\frac{\partial\rho}{\partial c}(c_\eps)\Je_\eps$, and all $\phie \in C^\infty_0([0,\infty); C_{0,\sigma}^\infty(\Omega)) $, as well as
  \begin{align}\label{eq:weakNSCH2}
    \int_Q c_\eps \left(\partial_t \psi +\Div (\psi \ve_\eps)\right)\sd (x,t) +
    \int_\Omega c_{0,\eps}\psi|_{t=0}\sd x=\int_Q {m_\eps(c_\eps)}\nabla\mu_\eps \cdot \nabla \psi \sd(x,t)
\end{align}
for all $\psi \in C_0^\infty([0,\infty)\times\ol{\Omega})$, and
\begin{align}\label{eq:weakNSCH3}
    \mu_\eps = \frac{f'(c_\eps)}{\eps} + \eps\, a'(c_\eps)\frac{|\nabla c_\eps|^2}{2} - \eps\, \Div (a(c_\eps)\nabla c_\eps)\quad &\text{ a.e. in } Q,\\ \no_{\partial\Omega}\cdot\nabla c_\eps = 0 \quad&\text{ a.e. on }(0,\infty)\times\partial\Omega.
% \int_Q \mu_\eps\,\psi \sd (x,t) = \int_Q \frac{f'(c_\eps)}{\epsilon}\,\psi + \eps\, a'(c_\eps)\frac{|\nabla c_\eps|^2}{2}\psi +  \eps\, a(c_\eps)\nabla c_\eps\cdot \nabla \psi\sd (x,t)
  \end{align}
   Moreover, we have
\begin{eqnarray}\nonumber
  \lefteqn{\int_\Omega\frac{\rho(c_\eps(t))|\ve_\eps(t)|^2}2\,dx +\E_\eps(c_{\eps}(t))
  +\int_s^t\int_\Omega \nu(c_\eps) |D\ve_\eps|^2 \sd x\sd \tau}\\\label{eq:enin'}
&&+\int_s^t\int_\Omega m_\eps(c_\eps) |\nabla
    \mu_\eps|^2\, d(x,t) \leq \int_\Omega\frac{\rho(c_\eps(s))|\ve_\eps(s)|^2}2\,dx +\E_\eps(c_{\eps}(s))
\end{eqnarray}
for all $t\geq s$ and almost every $s\geq 0$ including $s=0$.

\begin{theorem}\label{thm:SharpInterfaceLimit}
%   Let $\Omega\subset\R^d$, $d=2,3$, be a smooth, bounded domain and let $\nu\in
%   C^0(\R)$ with $\nu(s)\geq \nu_0 >0$ for all $s\in\R$.  
For all $\eps\in (0,1]$, let initial data $(\ve_{0,\eps},c_{0,\eps})\in
  L^2_\sigma(\Omega)\times H^1(\Omega)$ be given such that $\frac1{|\Omega|}\int_\Omega
  c_{0,\eps} \sd x=\bar c\in (-1,1)$ and
  \begin{equation}\label{eq:EstimIV}
    \int_\Omega \frac{|\ve_{0,\eps}|^2}{2} \sd x + \E_\eps(c_{0,\eps}) \leq R 
  \end{equation}
for some $R>0$. 
% Finally, let
% $(m_\eps)_{\eps\in (0,1]}$ with $m_\eps\to_{\eps\to 0} 0$ such that $\lim_{\eps} \frac{\eps}{m_\eps}=0$. (Note that this implies $m_\eps\geq \beta \eps$ for some $\beta>0$.)%%% Woanders hin
 Furthermore, let $(\ve_\eps,c_\eps,\mu_\eps)$ be weak solutions of (\ref{eq:NSCH1})-(\ref{eq:NSCH7}) in the interval $[0,\infty)$. 
Then there exists a sequence $(\eps_k)_{k\in\N}$, converging to $0$ as $k\to \infty$, such that the following assertions are true.
  \begin{enumerate}
\item There are $\ve\in L^2((0,\infty);H^1(\Omega)^d)\cap L^\infty((0,\infty);L^2_\sigma(\Omega)^d)$, $\ve_0\in L^2_\sigma(\Omega)$ such that, as $k\rightarrow\infty$,
  \begin{alignat}{2}
\label{eq:Conv2a}
    \ve_{\eps_k} &\rightharpoonup \ve &\quad& \text{in}\
    L^2((0,\infty);H^1(\Omega)^d),\\\label{eq:Conv2b}
    \ve_{\eps_k} &\rightarrow \ve &\quad& \text{in}\
    L^2_{loc}([0,\infty);L^2_\sigma(\Omega)),\\\label{eq:Conv2c}
   \ve_{0,\eps_k} &\rightharpoonup \ve_0 &\quad& \text{in}\
    L^2_\sigma (\Omega).
  \end{alignat} 
If $m_0>0$, there exists a $\mu\in L^2_{loc}([0,\infty);H^1(\Omega))$ with $\nabla\mu\in L^2((0,\infty);L^2(\Omega)^d)$ and such that
  \begin{alignat}{2}
\label{eq:Conv1}
    \mu_{\eps_k} &\rightharpoonup \mu &\qquad& \text{in}\
    L^2_{loc}([0,\infty);H^1(\Omega)).
  \end{alignat} 
  \item There are measurable sets $E\subset \Omega \times [0,\infty)$ and
    $E_0\subset \Omega$ such that, as $k\rightarrow\infty$,
  \begin{alignat}{2}\label{eq:Conv3}
    c_{\eps_k} &\to -1+2\chi_{E}&\quad& \text{a.e. in } \Omega \times (0,\infty) \text{ and in } C^{\frac19}_{loc}([0,\infty);L^2(\Omega))\\
    c_{0,\eps_k}&\to -1+2\chi_{E_0}&\quad& \text{a.e. in } \Omega. 
 \end{alignat}
In particular, we have $\chi_{E}|_{t=0}=\chi_{E_0}$ in $L^2(\Omega)$. 
\item There exist Radon measures $\RM$ and $\RM_{ij}$, $i,j=1,\ldots, d$ on
  $\ol{\Omega}\times [0,\infty)$ such that for every $T>0$, $i,j=1,\ldots, d$, as $k\rightarrow\infty$,
  \begin{alignat}{2}\label{eq:RM1}
    e_{\eps_k} (c_{\eps_k}) \sd x\sd t  &\rightharpoonup^\ast 
    \RM &\quad& \text{in}\ \mathcal{M}(\ol{\Omega}\times [0,T]),\\\label{eq:RM2}
    \eps_k\, a(c_{\eps_k})\,\partial_{x_i}c_{\eps_k}\partial_{x_j} c_{\eps_k} \sd x\sd t &\rightharpoonup^\ast 
    \RM_{ij} && \text{in}\ \mathcal{M}(\ol{\Omega}\times [0,T]).
  \end{alignat}
\item There exists a Radon measure $V= V^t dt$ on $\ol\Omega\times G_{d-1}\times
  (0,\infty)$ such that $(\ve,E,\mu,V)$ if $m_0>0$ and $(\ve,E,V)$ else is a varifold solution of
  (\ref{eq:1})-(\ref{eq:6}) in the sense of
  Definition~\ref{defn:VarifoldSolution} with initial values $(\ve_0,E_0)$ and $\sigma= \int_{-1}^1 \sqrt{f(s)/2}\sd s$. Furthermore,
  \begin{equation}\label{eq:ReprFirstVar}
    \int_0^T \weight{\delta V^t,\etae} \sd t = \int_0^T \int_\Omega
    \nabla \etae : \left(\sd \RM\, I- (d\RM_{ij})_{i,j=1}^d\right)\, dt
  \end{equation}
  for all $\etae \in C^1_0(\Omega \times [0,T];\R^d)$. %%% Definition!!!
  \item If $\ve_{0,\eps_k}\rightarrow\ve_0$ in $L^2_\sigma(\Omega)$ and $\E_\eps(c_{0,\eps})\rightarrow 2\sigma|\nabla\chi_{E_0}|(\Omega)$ as $k\to\infty$, then \eqref{eq:EnergyEstim} holds for almost all $t\in (0,\infty)$, $s=0$, and $\lambda^0(\overline\Omega)$ replaced by $2\sigma|\nabla\chi_{E_0}|(\Omega)$.
  \end{enumerate}
\end{theorem}
   By \eqref{eq:enin'} and the assumptions on the initial data we obtain
\begin{eqnarray}\nonumber
  \lefteqn{\int_\Omega\frac{\rho(c_\eps(t))|\ve_\eps(t)|^2}2\,dx +\E_\eps(c_{\eps}(t))}\\\label{eq:enin}
&&  +\int_0^t\int_\Omega \nu(c_\eps) |D\ve_\eps|^2 +m_\eps(c_\eps) |\nabla
    \mu_\eps|^2\, dx\sd t \leq R 
\end{eqnarray}
for all $t\ge 0$. 

From this estimate, Korn's inequality, and \eqref{eq:EstimIV} we deduce that there exists a sequence $\eps_k\searrow 0$ as $k\to\infty$ such that \eqref{eq:Conv2a},\eqref{eq:Conv2c},\eqref{eq:Conv1},\eqref{eq:RM1}, and \eqref{eq:RM2} hold. Using the assumptions on $f$, we further deduce that
\begin{align}\label{eq:cest}
  \int_\Omega |c_\eps(t)|^p \sd x &\leq C(1+R),\\ \label{eq:cest2}
\int_\Omega (|c_\eps(t)|-1)^2\, dx &\leq C\eps R 
\end{align}
for all $t\ge 0$. In particular, for \eqref{eq:cest2} we used that $f(c)\ge C (|c|-1)^2$ for all $c\in\R$ and some constant $C>0$ which follows from the positivity of $f''(\pm 1)$ and the $p$-growth of $f$ for large $c$. With the definitions (cf. \cite{ChenLimitCH})
\[W(c)=\int_{-1}^c\sqrt{2\tilde{f}(s)}\sd s,\ \text{ where } \tilde{f}(s)= \min (f(s),1+|s|^2),\]
and
\[w_\eps (x,t) = W(c_\eps(x,t)),\] 
the functions $w_\eps$ are uniformly bounded in
$L^\infty((0,\infty);BV(\Omega))$ since
\begin{equation}\label{eq:NablaWepsEstim}
  \int_\Omega |\nabla w_\eps(x,t)|\sd x = \int_\Omega
  \sqrt{2\tilde{f}(c_\eps(x,t))}|\nabla c_\eps(x,t)|\sd x \leq \int_\Omega e_\eps
  (c_\eps(x,t))\sd x \leq R.
\end{equation}
Moreover, note that by the assumptions on $f$, there exist constants $C_0,C_1>0$ such that for all $c_0,c_1\in\R$
\begin{equation}\label{eq:WEstim}
  C_0|c_0-c_1|^2 \leq |W(c_0)-W(c_1)|\leq C_1 |c_0-c_1| (1+|c_0|+|c_1|).
\end{equation}
Here, for the first inequality we used again that $f(s)\ge C (|s|-1)^2$ for all $s\in\R$.

\begin{lem}\label{lem:ctest}
  There exists a constant $C>0$ such that
  \begin{equation*}
    \|w_\eps\|_{C^{\frac18}([0,\infty); L^1(\Omega))}+
    \|c_\eps\|_{C^{\frac18}([0,\infty); L^2(\Omega))} \leq C.
  \end{equation*}
\end{lem}
\begin{proof}
  The proof is a modification of \cite[Proof of Lemma~3.2]{ChenLimitCH}. Therefore, we only
  give a brief presentation. For sufficiently small $\eta>0$, $ x\in\Omega$, and $t\geq 0$ let
  \begin{equation*}
    c_\eps^\eta (x,t) = \int_{B_1} \omega(y)\, c_\eps (x-\eta y,t) \sd y,
  \end{equation*}
  where $\omega$ is a standard mollifying kernel and $c_\eps$ is extended to a small
  neighborhood of $\Omega$ as in \cite[Proof of Lemma~3.2]{ChenLimitCH}. Then, there exist constants $C,C'>0$ such that
  \begin{align}\label{eq:gradest}
    \|\nabla c^\eta_\eps(t)\|_{L^2(\Omega)} &\leq
    C\eta^{-1}\|c_\eps(t)\|_{L^2(\Omega)}\leq C'\eta^{-1}\\\label{eq:diffest}
    \|c^\eta_\eps(t) - c_\eps(t)\|_{L^2(\Omega)}^2&\leq C\eta \|\nabla
    w_\eps(t)\|_{L^1(\Omega)}\leq C'\eta 
  \end{align}
for all sufficiently small $\eta>0$, cf. \cite[Proof of Lemma~3.2]{ChenLimitCH}.
% Next we use that
% \begin{equation*}
%   (c_\eps(\cdot,t) - c_\eps(\cdot,\tau),\varphi)_\Omega = -\int_\tau^t \left((\nabla
%     \mu- v_\eps
%     c_\eps)(s), \nabla \varphi \right)_\Omega \sd s
% \end{equation*}
% for all $\varphi \in C^1(\ol{\Omega})$
% because of (\ref{eq:NSCH3}) in its weak form.  %%% Genauer???
From \eqref{eq:gradest} and \eqref{eq:weakNSCH2} we deduce that for all $0\le\tau< t<\infty$ such that $|t-\tau|\le 1$
\begin{align}\nonumber
  &\int_\Omega (c_\eps(x,t)- c_\eps
  (x,\tau))(c_\eps^\eta(x,t)- c_\eps^\eta (x,\tau)) \sd x=\\\nonumber
&-\int_\tau^t\int_\Omega \Big({m_\eps(c_\eps(x,s))} \nabla\mu_\eps(x,s)- \ve_\eps(x,s)c_\eps(x,s)\Big)\cdot\Big(\nabla c_\eps^\eta(x,t)- \nabla c_\eps^\eta (x,\tau)\Big)
  \, d(x,s)\\\label{eq:timeest}
& \leq C(R)(t-\tau)^{\frac12} \sup_{s\in (\tau,t)}\|\nabla c_\eps^\eta
(s)\|_{L^2(\Omega)} \leq C(R)\eta^{-1}(t-\tau)^{\frac12}.
\end{align}
Here, we used the fact that for all $\tau,t$ as above we have
\begin{equation*}
  \big\|{m_\eps(c_\eps)}\nabla\mu_\eps- \ve_\eps c_\eps\big\|_{L^2(\Omega\times (\tau,t))} \leq C(R),
\end{equation*}
since the sequences $(\ve_\eps) \subset L^2((0,\infty);L^6(\Omega))$ and $(c_\eps) \subset
L^\infty((0,\infty);L^3(\Omega))$ are bounded due to the assumptions $d\leq 3$ and $p\geq
3$. Now, combining \eqref{eq:timeest}, \eqref{eq:diffest} and using H\"older's and Young's inequality we conclude that for $\eta,\tau,$ and $t$ as above we have
\begin{equation*}
  \|c_\eps (t)-c_\eps(\tau)\|_{L^2(\Omega)}^2 \leq C(\eta+\eta^{-1}|t-\tau|^{\frac12}).
\end{equation*}
Choosing $\eta=(t-\tau)^\frac14$ for sufficiently small $t-\tau$ we conclude the claim concerning $c_\eps$. Using (\ref{eq:WEstim}), one derives the claim concerning $w_\eps$
as in \cite{ChenLimitCH}.
\end{proof}

\begin{rem}
It is possible to understand the proof of Lemma \ref{lem:ctest} from a more general point of view. From \eqref{eq:weakNSCH2} and \eqref{eq:enin} we easily deduce that the distributional time-derivative of $(c_\eps)$ is uniformly bounded in $L^2((0,\infty),H^{-1}(\Omega))$. In particular, $(c_\eps)$ is uniformly bounded in $C^{1/2}([0,\infty),H^{-1}(\Omega))$. On the other hand, the computations leading to \eqref{eq:diffest} show that $(c_\eps)$ is uniformly bounded in $L^\infty((0,\infty);B^{1/3}_{2\infty}(\Omega))$. This follows from $B^{1/3}_{2\infty}(\Omega)= (L^2(\Omega),H^1(\Omega))_{1/3,\infty}$ and the definition of the real interpolation spaces with the aid of the $K$-method. By interpolation, we obtain uniform boundedness in $C^{1/8}([0,\infty),L^2(\Omega))$.
\end{rem}

The proof of the following lemma is literally the same as the proof of \cite[Lemma~3.3]{ChenLimitCH}.
\begin{lem}\label{lem:cconv}
  There exists a subsequence (again denoted by $\eps_k$) and a measurable set $E\subset
  \Omega\times [0,\infty)$ such that, as $k\to\infty$,
  \begin{align*}
   w_{\eps_k}\to 2 \STC \chi_E&\quad\text{ a.e. in }\Omega\times
    (0,\infty)\text{ and in }C^{\frac19}_{loc}([0,\infty);L^1(\Omega))\\ %%% t=0???
  c_{\eps_k}\to -1+2\chi_E&\quad\text{ a.e. in }\Omega\times
    (0,\infty)\text{ and in }C^{\frac19}_{loc}([0,\infty);L^2(\Omega))
  \end{align*}
Moreover, $\chi_E\in L^\infty_{\omega\ast}((0,\infty);BV(\Omega))\cap C^{\frac14}([0,\infty);L^1(\Omega))$ and for all $t\ge 0$ we have $|E_t|= |E_0|= \frac{1+\bar{c}}2|\Omega|$.
\end{lem}

\begin{lem}\label{lem:muest}
There exist constants $C,\eps_0>0$ such that
\begin{equation}
  \label{eq:H1ChemPotEstim}
  \|\mu_\eps (t)\|_{H^1(\Omega)} \leq C \left(\E_\eps(c_\eps(t)) + \|\nabla \mu_\eps(t)\|_{L^2(\Omega)}\right)
\end{equation}
for almost all $t>0$ and $0<\eps\leq \eps_0$. Using $m_\eps\ge \overline m_\eps$ we deduce from \eqref{eq:H1ChemPotEstim} and \eqref{eq:enin} that
\begin{equation}
  \label{eq:L2ChemPotEstim}
  \overline m_\eps\int_0^T\|\mu_\eps (t)\|_{L^2(\Omega)}^2\, dt \leq C(R,T)\quad\text{ for all } 0<T<\infty.
\end{equation}
\end{lem}
\begin{proof}
  Let us suppress the time variable. Due to Poincar\'e's inequality it suffices to control the average of $\mu_\eps$. Equation \eqref{eq:weakNSCH3} can be written in the form
\begin{equation}\label{eqn:elliptic}
 \mu_\eps=\frac{f'(c_\eps)}{\eps}-\eps\sqrt{a(c_\eps)}\,\Delta A(c_\eps).
\end{equation}
Multiplying by $\etae\cdot\nabla c_\eps$ for $\etae\in C^1(\ol\Omega;\R^d)$, integrating over $\Omega$, and integrating by parts yields
\begin{align}\nonumber
 \int_\Omega \etae\cdot\nabla c_\eps\,\mu_\eps\, dx=-\int_\Omega \nabla\etae:\big(e_\eps(c_\eps)\,I-\eps\,\nabla A(c_\eps)\otimes\nabla A(c_\eps)\big)\, dx\\\label{eq:fvare} +\int_{\partial\Omega} e_\eps(c_\eps)\,\etae\cdot\no_{\partial\Omega}\,d\HS^{d-1}.
\end{align}
Now we can proceed exactly as in the proof of \cite[Lemma~3.4]{ChenLimitCH}.
\end{proof}

\begin{lem}\label{lem:vest}
 There exists a subsequence (again denoted by $\eps_k$) such that, as $k\to\infty$,
\begin{alignat*}{2}
\ve_{\eps_k} &\to \ve &&\quad \text{ in } L^2_{loc}([0,\infty);L^2_\sigma(\Omega))\\
\ve_{\eps_k}(t) &\to \ve(t)&&\quad \text{ in } L^2(\Omega) \text{ for almost every } t> 0.
\end{alignat*}
Furthermore, there exists a measurable, non-increasing function $\E(t)$, $t>0$, such that for almost all $t>0$
\begin{align}\label{eqn:en}
\E_{\eps_k}(c_{\eps_k}(t))\to \E(t)\quad\text{ and }\quad|\nabla \chi_{E_t}|(\Omega) \leq \frac1{2\STC} \mathcal{E}(t)\leq\frac1{2\STC} R.
\end{align}
\end{lem}
\begin{proof}
Let us fix some $T>0$ and let $P_\sigma:L^2(\Omega)^d\rightarrow L^2_\sigma(\Omega)$ denote the Helmholtz projection. In order to prove the claim concerning $\ve_{\eps_k}$ it suffices to show that for a subsequence we have 
\begin{equation}\label{eq:stconv}
P_\sigma(\rho(c_{\eps_k})\ve_{\eps_k})\rightarrow_{k\rightarrow\infty}P_\sigma(\rho(c)\ve)
\quad\text{ in }L^2((0,T);(L^2_\sigma(\Omega)\cap H^1(\Omega)^d)')
\end{equation}
since then
\begin{align*}
\int_0^T\int_\Omega\rho(c_{\eps_k})|\ve_{\eps_k}|^2\, d\sd t=\int_0^T\int_\Omega P_\sigma(\rho(c_{\eps_k})\ve_{\eps_k})\cdot\ve_{\eps_k}\, dx\sd t\\
\rightarrow_{k\to\infty} \int_0^T\int_\Omega P_\sigma(\rho(c)\ve)\cdot\ve\, dx\sd t=\int_0^T\int_\Omega\rho(c)|\ve|^2\, dx\sd t,
\end{align*}
 and from this convergence, the strong convergence of $(c_{\eps_k})$, and the strict positivity of $\rho$ we easily deduce the claim, cf. \cite{ModelH}. But \eqref{eq:stconv} follows from the Aubin-Lions lemma by noting that, firstly, 
\[L^2_\sigma(\Omega)\hookrightarrow\hookrightarrow(L^2_\sigma(\Omega)\cap H^1(\Omega)^d)'\hookrightarrow (L^2_\sigma(\Omega)\cap W^{1,\infty}(\Omega))'\]
and that, secondly, the distributional time-derivative of $(P_\sigma(\rho(c_{\eps_k})\ve_{\eps_k}))$ is uniformly bounded in $L^{8/7}((0,T);(L^2_\sigma(\Omega)\cap W^{1,\infty}(\Omega))')$. This last bound follows by estimating each term in \eqref{eq:weakNSCH1}. We have (appreviating $L^p((0,T);L^{q}(\Omega))$ by $L^p L^q$)
\begin{align*}
\|\rho(c_\eps)\ve_\eps\otimes\ve_\eps\|_{L^2L^{3/2}}&\le \|\rho(c_\eps)\ve_\eps\|_{L^\infty L^{2}}\|\ve_\eps\|_{L^2L^{6}},\\
\|\ve_\eps\otimes \tilde\Je_\eps\|_{L^{8/7}L^{4/3}}&\le \|\ve_\eps\otimes \tilde\Je_\eps\|_{L^{1}L^{3/2}}^{3/4}\|\ve_\eps\otimes \tilde\Je_\eps\|_{L^{2}L^{1}}^{1/4}\\
&\le C \|\ve_\eps\|_{L^{2}L^{6}}^{3/4}\|\ve_\eps\|_{L^{\infty}L^{2}}^{1/4}\|m(c_\eps)|\nabla\mu_\eps|^2\|_{L^{1}L^{1}},\\
\|\nu(c_\eps)D\ve_\eps\|_{L^{2}L^{2}}&\le C\|D\ve_\eps\|_{L^{2}L^{2}},\\
\|\eps\, a(c_\eps)\nabla c_\eps\otimes\nabla c_\eps\|_{L^{\infty}L^{1}}&\le C\|\eps|\nabla A(c_\eps)|^2\|_{L^{\infty}L^{1}}.
\end{align*}
Concerning the remaining claims we note that the total energies
   \begin{equation*}
     \E^{tot}_\eps(t):= \frac12 \|v_\eps(t)\|_{L^2(\Omega)}^2 + \E_{\eps}(c_\eps(t)),\quad t\geq 0,
   \end{equation*}
form a sequence of bounded, non-increasing functions and that $\ve_{\eps_k}(t)\to_{k\to\infty} \ve(t)$ for almost all $t>0$ in $L^2(\Omega)$. Now, we can proceed exactly as in the proof of \cite[Lemma~3.3]{ChenLimitCH}.
\end{proof}

Finally, we define the discrepancy function by
\[\xi^\eps(c_\eps):= \frac{\eps}2|\nabla A(c_\eps)|^2 -\frac1\eps f(c_\eps).\]
\begin{theorem}\label{thm:disc}
For all sufficiently small $\eta>0$ there exists a constant $C(\eta)$ such that for all sufficiently small $\eps>0$ (the maximal $\eps$ may depend on $\eta$) we have
\begin{equation*}
\int_{0}^T\int_\Omega (\xi^\eps(c_\eps))^+\, d(x,t)\leq \eta \int_0^T\int_{\Omega} e_\eps(c_\eps)\, d(x,t) +\eps\, C(\eta)\int_0^T \int_\Omega|\mu_\eps|^2\, d(x,t).
\end{equation*} 
Combining this estimate with the assumption $\eps/\overline m_\eps\rightarrow_{\eps\rightarrow 0}0$ and \eqref{eq:L2ChemPotEstim} we deduce that
\begin{equation*}
  \lim_{\eps \to 0} \int_{0}^T\int_\Omega (\xi^\eps(c_\eps))^+\, d(x,t)=0\qquad \text{for all}\ 0<T<\infty.
\end{equation*}
\end{theorem}
\begin{proof}
The proof is based on the elliptic equation \eqref{eqn:elliptic} which can be written in the form
\begin{equation*}
 \mu_\eps\, a(A(c_\eps))^{-1/2}=\frac{(f\circ A^{-1})'(A(c_\eps))}{\eps}-\eps\,\Delta A(c_\eps).
\end{equation*}
Let $c_{\pm}:=A(\pm 1)$, $B(c):=c\frac{c_+-c_-}{2}+\frac{c_+ + c_-}{2}$, and $\tilde f(c):=f(A^{-1}(B(c)))/(B')^2$ for $c\in\R$. Then $\tilde f$ fulfills Assumption \ref{assumptions}, and for $\tilde c_\eps:=B^{-1}(A(c_\eps))$ we have
\begin{equation*}
 \mu_\eps\, a(A(c_\eps))^{-1/2}(B')^{-1}=\frac{\tilde f'(\tilde c_\eps)}{\eps}-\eps\,\Delta \tilde c_\eps.
\end{equation*}
Since the function $a(A(c_\eps))^{-1/2}(B')^{-1}$ is uniformly bounded, \cite[Theorem~3.6]{ChenLimitCH} yields 
\begin{equation}\label{eqn:discr}
\int_{0}^T\int_\Omega (\tilde\xi^\eps(c_\eps))^+\, d(x,t)\leq \eta \int_0^T\int_{\Omega} \tilde e_\eps(c_\eps)\, d(x,t) +\eps\, C(\eta)\int_0^T \int_\Omega|\mu_\eps|^2\, d(x,t)
\end{equation} 
where 
\begin{align*}
\tilde\xi^\eps(\tilde c_\eps)&:=\frac{\eps}2|\nabla\tilde c_\eps|^2 -\frac1\eps \tilde f(\tilde c_\eps)=\xi^\eps(c_\eps)/(B')^2\\
\tilde e^\eps(\tilde c_\eps)&:=\frac{\eps}2|\nabla\tilde c_\eps|^2 +\frac1\eps \tilde f(\tilde c_\eps)=e^\eps(c_\eps)/(B')^2.
\end{align*}
This proves the claim.
\end{proof}

Using the previous statements, we can now easily finish the proof of
Theorem~\ref{thm:SharpInterfaceLimit} by the arguments of
\cite[Section~3.5]{ChenLimitCH}. To be more precise, item 1 follows from \eqref{eq:EstimIV} and Lemmas \ref{lem:muest} and \ref{lem:vest}. Item 2 follows from Lemma \ref{lem:cconv} and the energy inequality \eqref{eq:enin}. Item 3 follows from \eqref{eq:enin} as well. Furthermore, we note that $\RM=\RM^t dt$ for Radon measures $\RM^t$ on $\ol{\Omega}$ since
\begin{equation*}
  \lambda (A\times I) \leq \lambda (\ol{\Omega}\times I)= \lim_{k\to\infty}\int_I \E_{\eps_k}(\tau)\sd \tau \leq |I|R
\end{equation*}
for any measurable $A\subseteq \ol{\Omega}, I\subseteq [0,\infty)$.
Similarly, we also get $\lambda^t(\ol\Omega)=\E(t)$ for almost all $t\in(0,\infty)$ due to \eqref{eqn:en}. From \eqref{eq:enin'} we deduce that
\begin{align*}
  \RM^t(\ol{\Omega}) &= \lim_{k\to\infty} \E_{\eps_k}(c_{\eps_k}(t)) \\
  &\leq -\liminf_{k\to\infty} \int_s^t
    \int_\Omega \left(\nu(c_{\eps_k}) |D \ve_{\eps_k}|^2 + m_{\eps_k}(c_{\eps_k})
      |\nabla \mu_{\eps_k}|^2\right) d(x,\tau)\\
  &\quad + \lim_{k\to \infty} \left(\E_{\eps_k}(c_{\eps_k}(s))+\frac12\int_\Omega\rho(c_{\eps_k}(s))|\ve_{\eps_k}(s)|^2\,dx - \frac12\int_\Omega\rho(c_{\eps_k}(t))|\ve_{\eps_k}(t)|^2\,dx\right)\\
  &\leq - \int_s^t
  \int_\Omega \left(\nu(c) |D \ve|^2\sd x - \Je\cdot\nabla\mu\right)d(x,\tau)+\RM^s(\ol{\Omega})\\
&\quad +\frac12\int_\Omega\rho(c(s))|\ve(s)|^2\,dx
    -\frac12\int_\Omega\rho(c(t))|\ve(t)|^2\,dx
\end{align*}
for almost all $0 < s < t<\infty$ where $c:=-1+2\chi_{E}$. This is \eqref{eq:EnergyEstim}. Item 5 follows similarly. We can proceed as in \cite[Section~3.5]{ChenLimitCH} to construct the varifold $V$. Therefore, we only give a sketch. We deduce from Theorem \ref{thm:disc} that for all $\etae_0,\etae_1\in C^1(\ol\Omega;\R^d)$ and all $0<T<\infty$
\[\int_0^T\int_{\ol{\Omega}}\etae_0\otimes \etae_1:(d\lambda_{ij})\le\int_0^T\int_{\ol{\Omega}}|\etae_0||\etae_1|\ d\lambda.\]
This proves the existence of $\lambda$-measurable $\R$-valued, non-negative functions $\gamma_i$  and $\lambda$-measurable unit vector fields $\nue_i$, $i=1,\ldots,d$, such that 
\begin{align*}
(\lambda_{ij})=\sum_{i=1}^d\gamma_i\,\nue_i\otimes\nue_i\,\lambda \quad\text{ and }\quad
\sum_{i=1}^d\gamma_i\le 1,\ \sum_{i=1}^d\nue_i\otimes\nue_i=I\quad\lambda\text{-a.e.}
\end{align*}
We denote the equivalence class of $\nue_i(x,t)$ in $G_{d-1}$ by $\pe_i^t(x)$, define the functions $b_i^t$ by
\[b_i^t(x):=\gamma_i(x,t)+\frac{1}{d-1}\Big(1-\sum_{i=1}^d\gamma_i(x,t)\Big)\]
and define the varifold $V$ as in \eqref{eqn:V}. Then item 3 in Definition \ref{defn:VarifoldSolution} follows taking into account \eqref{eqn:en}. Furthermore, in the case $m_0>0$ we infer from \eqref{eq:fvare} that 
\begin{align*}
 \int_\Omega 2\chi_{E_t}\Div(\mu\etae)\, dx=\int_\Omega \nabla\etae:\big(d\lambda\, I-(d\lambda_{ij})_{i,j=1}^d\big)&=\int_\Omega \nabla\etae:\sum_{i=1}^d b_i^t\big(I-\pe_i^t\otimes\pe_i^t\big)\, d\lambda\\
&=\left\langle \delta V^t,\etae\right\rangle
\end{align*}
for all $\etae\in C^1_0(\ol\Omega;\R^d)$ and almost all $t\in (0,\infty)$. This is \eqref{eq:fvar}. Furthermore, these calculations prove \eqref{eq:ReprFirstVar}. Similarly, \eqref{eq:ns'} and \eqref{eq:mu'} follow from  \eqref{eq:weakNSCH1} and \eqref{eq:weakNSCH2}, respectively, where one uses that 
\begin{equation*}
  \int_Q\eps a(c_\eps)\nabla c_\eps \otimes \nabla c_\eps : \nabla \phie \sd (x,t)= \int_Q \phie\cdot\nabla c_\eps \mu_\eps \sd (x,t) \to_{\eps\to 0} \left\langle \delta V^t,\phie\right\rangle
\end{equation*}
for all $\phie\in C^\infty([0,\infty);C_0^\infty(\Omega))$.
This proves item 4 in Definition \ref{defn:VarifoldSolution}. Finally, item 2 in Definition \ref{defn:VarifoldSolution} follows from Lemma \ref{lem:cconv}. This concludes the proof of Theorem \ref{thm:SharpInterfaceLimit}.

In the radially symmetric case we can prove a stronger statement concerning the discrepancy measure. 

\begin{theorem}\label{thm:Radial}
Let $\Omega=B_1(0)$, and assume that the solutions $(\ve_\eps,c_\eps,\mu_\eps)$ are radially symmetric. Assume, furthermore, that $A(c)=c$ for all $c\in\R$, and that the constants $\ol{m}_\eps$ in the Assumptions \ref{assumptions} satisfy 
\begin{equation}
  \label{eq:ConMeps}
\eps^{\frac{1}{d-1}}/\ol{m}_\eps\to_{\eps\to 0}0.  
\end{equation}
  Then, for all $T>0$, we have
\[\lim_{\eps\searrow 0}\int_{0}^T\int_\Omega |\xi^\eps(c_\eps)|\, d(x,t)=0.\]
\end{theorem}
For the proof we need the following result from \cite[Lemma~4.4]{ChenLimitCH}.
\begin{lem}
 There exist positive constants $C_0$ and $\eta_0$ such that for every $\eta\in [0,\eta_0]$, $\eps\in (0,1]$, and every $(u^\eps,v^\eps)\in H^2(\Omega)\times L^2(\Omega)$ such that
  \begin{equation*}
    v^\eps = -\eps \Delta u^\eps +\eps^{-1}f'(u^\eps),\qquad \no_{\partial\Omega}\cdot \nabla u^\eps|_{\partial\Omega} =0
  \end{equation*}
we have
\begin{eqnarray}\nonumber
\lefteqn{ \int_{\{x\in\Omega:u^\eps|\geq 1-\eta\}}\left(e^\eps(u^\eps)+\eps^{-1}(f'(u^\eps))^2\right)}\\\label{eq:Lem44Estim}
&\leq& C_0\eta \int_{\{x\in\Omega:|u^\eps|\leq 1-\eta\}}\eps |\nabla u^\eps|^2\sd x + C_0\eps\int_\Omega |v^\eps|^2 \sd x
\end{eqnarray}
\end{lem}
\noindent
\begin{proof*}{of Theorem~\ref{thm:Radial}}
 We can show exactly like in \cite[Proof of Theorem 5.1]{ChenLimitCH} that there exists a constant $C>0$ such that for almost all $t>0$ we have
\begin{alignat}{2}\label{eq:deltaest}
 \int_{B_\delta} e_\eps(c_\eps(t))\, dx&\le C\delta M^\eps(t)&\quad&\text{ for all }\delta\in (0,1),\\\label{eq:rest}
|\xi^\eps(c_\eps(r,t))+\mu_\eps(r,t) c_\eps(r,t)|&\le C r^{1-d} M^\eps(t)&\quad&\text{ for all }r\in (0,1).
\end{alignat}
Here, we use the notation $r=|x|$ and
\[M^\eps(t):=1+\|\mu_\eps(t)\|_{H^1(\Omega)}+\eps \|\mu_\eps(t)\|_{H^1(\Omega)}^2.\] From \eqref{eq:rest} we deduce that for small $\delta,\eta>0$
\begin{align*}
 \int_\Omega|\xi^\eps(c_\eps(t))|\, dx \le&\int_{B_\delta\cup\{|c_\eps(t)|\ge 1-\eta\}}e_\eps(c_\eps(t))\, dx+ \int_{\Omega\cap\{r>\delta,|c_\eps(t)|< 1-\eta\}}|\mu_\eps(t)|(1-\eta)\, dx\\ 
&+ CM^\eps(t)\int_{\Omega\cap\{r>\delta,|c_\eps(t)|< 1-\eta\}} r^{1-d}\, dx.
\end{align*}
Using \eqref{eq:deltaest} and \eqref{eq:Lem44Estim}, the first integral on the right hand side may be estimated by
\[C\delta M^\eps(t) + C'\eta \E_\eps(t)+C'\eps\|\mu_\eps(t)\|_{L^2(\Omega)}^2.\]
By \eqref{eq:cest2}, the second integral is dominated by
\[\|\mu_\eps(t)\|_{L^2(\Omega)}\,|\{|c_\eps(t)|\ge 1-\eta\}|^{1/2}\le C''(\eta)M^\eps(t)\eps^{1/2}.\]
Finally, using \eqref{eq:Lem44Estim} again, the third integral is smaller than $C'(\eta)M^\eps(t)\delta^{1-d}\eps$. Summing up, we have
\[\int_\Omega|\xi^\eps(c_\eps(t))|\, dx\le C'\eta \E_\eps(t)+C'\eps\|\mu_\eps(t)\|_{L^2(\Omega)}^2+C''(\eta)M^\eps(t)(\eps^{1/2}+\delta^{1-d}\eps+\delta).\]
Integrating this estimate from $0$ to $T$ and choosing $\eta$ small, the first term on the right hand side gets arbitrarily small. Choosing then $\delta=\eps^{1/(2d-2)}$ and $\eps$ small the other two terms get arbitrarily small, too. While this is obvious for second term, concerning the third term we remark that it takes the form
\[C''(\eta)\int_0^T M^\eps(t)\,dt\, (\eps^{1/2}+\eps^{1/(2d-2)})=o(1)\qquad \text{as}\ \eps\to 0\]
due to \eqref{eq:ConMeps}.
\end{proof*}

%\label{app:}

%%% Local Variables:
%%% mode: latex
%%% TeX-master: "ChenSharpInterfaceLimit.tex"
%%% End:

\section{Nonconvergence}\label{sec:NonConvergence}

In this section we show that solutions of \eqref{eq:NSCH1}-\eqref{eq:NSCH4} do not converge in general to solutions of \eqref{eq:1}-\eqref{eq:4}  if $m_\eps(c)= \tilde{m}\eps^\alpha$ for some $\alpha>3$ or $m_\eps(c)\equiv 0$, which corresponds to the case ``$\alpha=\infty$''. More precisely, we will determine radially symmetric solutions which converge as $\eps\to 0$ to a solution, which does not satisfy \eqref{eq:3}. Moreover, for these solutions the discrepancy measure $\xi_\eps(c_\eps)$ does not vanish in the limit $\eps\to 0$. 

For simplicity of  the following presentation we assume that $\nu(c)\equiv \nu$, $\rho(c)\equiv \rho$.
We will construct radially symmetric solutions of the form 
\begin{equation}
  \label{eq:RadiallySolutions}
\ve(x,t)= u(r,t)\vc{e}_r,\quad p(x,t)= \tilde{p}_\eps(r,t),\quad c(x,t)= \tilde{c}_\eps(r,t),\quad \mu(x,t)= \tilde{\mu}_\eps(r,t),  
\end{equation}
 where $r=|x|, \vc{e}_r=\frac{x}{|x|}$.
If $(\ve,p,c,\mu)$ are of this form, \eqref{eq:NSCH1}-\eqref{eq:NSCH4} reduce to
\begin{eqnarray}\nonumber
  \rho\partial_t u + \rho u\partial_r u & -& \nu \tfrac1{r^{n-1}}\partial_r \left(r^{n-1}\partial_r u\right) \\\label{eq:RotNSCH1}
+ \partial_r \tilde{p}_\eps &=& - \eps \tfrac{n-1}{r}|\partial_r \tilde{c}_\eps|^2 -\eps \partial_r |\partial_r \tilde{c}_\eps|^2\\\label{eq:RotNSCH2}
\partial_r (r^{n-1} u) &=&0\\\label{eq:RotNSCH3}
\partial_t \tilde{c}_\eps + u\partial_r\tilde{c}_\eps &=& m_0\eps^\alpha \tfrac1{r^{n-1}}\partial_r \left(r^{n-1}\partial_r \tilde{\mu}_\eps\right)\\\label{eq:RotNSCH4}
\tilde{\mu}_\eps&=& -\eps  \tfrac1{r^{n-1}}\partial_r \left(r^{n-1}\partial_r \tilde{c}_\eps\right)+\eps^{-1}f'(\tilde{c}_\eps).
\end{eqnarray}
Here we have used 
\begin{eqnarray*}
  -\eps \Div \left(\nabla c\otimes \nabla c\right) &=& -\eps \Div \left(|\partial_r \tilde{c}_\eps|^2\vc{e}_r\otimes \vc{e}_r\right)\\
&=& -\eps (n-1)\tfrac1r |\partial_r \tilde{c}_\eps|^2\vc{e}_r - \eps \partial_r |\partial_r \tilde{c}_\eps|^2 \vc{e}_r
\end{eqnarray*}
since 
\begin{equation*}
  \nabla \vc{e}_r = \frac1r (I-\vc{e}_r\otimes \vc{e}_r),\qquad \Div \vc{e}_r = \tfrac{n-1}r. 
\end{equation*}
    We note that because of \eqref{eq:RotNSCH2}  $u(r,t)\equiv ar^{-n+1}\vc{e}_r$ for some $a\in\R$, which will be determined by the boundary conditions in the following. Hence we can solve \eqref{eq:RotNSCH3}-\eqref{eq:RotNSCH4} together with suitable boundary conditions and $c|_{t=0}= c_{0,\eps}$ independently and use \eqref{eq:RotNSCH1} afterwards to determine $\tilde{p}_\eps$.

\subsection{Nonconvergence in the Case $\alpha=\infty$}
First we consider the case $m_0=0$ (resp. ``$\alpha=\infty$''). In this case we consider \eqref{eq:RotNSCH1}-\eqref{eq:RotNSCH4} in the domain $\Omega = \{x\in \Rn: |x|>1\}$ together with the inflow boundary condition
\begin{alignat}{2}\label{eq:RotBC1}
  u(1,t) &= a &\qquad& \text{for all}\ t>0,\\\label{eq:RotBC2}
  \tilde{c}_\eps(1,t)&= 1&\qquad& \text{for all}\ t>0
\end{alignat}
for some $a>0$
and the initial values 
\begin{equation*}
  (u,\tilde{c})|_{t=0} = \left(\tfrac{a}{r^{n-1}}, c_{0,\eps}\right).
\end{equation*}
Here \eqref{eq:RotNSCH2} and \eqref{eq:RotBC1} already determine $u$ uniquely as
\begin{equation}\label{eq:ExplicitU}
  u(r,t) = \frac{a}{r^{n-1}} \qquad \text{for all}\ r\geq 1, t>0.
\end{equation}
Moreover, we choose 
\begin{equation}\label{eq:RotIV}
  \tilde{c}_{0,\eps}(r)= \theta\left( \frac{r-r_0}\eps \right) \qquad \text{for all}\ r\geq 1
\end{equation}
for some $r_0>1$, where
\begin{equation}
  \label{eq:theta}
 \theta \in C^\infty(\R)\quad \text{ such that } \quad \theta(s) =
 \begin{cases}
   1&\text{if}\ s<-\delta,\\
-1 &\text{if}\ s>\delta
 \end{cases}
\end{equation}
and $\delta\in (0,r_0-1)$ and $\eps \in (0,1]$.
Hence $\tilde{c}$ is a solution of the transport equation
\begin{alignat*}{2}
  \partial_t \tilde{c}_\eps(r,t) + \tfrac{a}{r^{n-1}}\partial_r \tilde{c}_\eps(r,t)&=0 &\qquad& \text{for}\ r>1,t>0,\\
\tilde{c}_\eps(1,t)&= 1 &&\text{for}\ t>0,
\end{alignat*}
which can be calculated with the method of characteristics. The solution for the initial condition above is 
\begin{equation}\label{eq:LimitSolution2}
\tilde{c}_\eps^\infty(r,t):=  \tilde{c}_\eps(r,t)=
  \begin{cases}
c_{0,\eps} (\sqrt[n]{r^n-ant})&\quad \text{if}\ r^n\geq ant,\\
1 &\quad \text{if}\ r^n < ant.
  \end{cases}
\end{equation}
By the construction we have
\begin{equation}
  \tilde{c}_\eps (r,t)\to_{\eps\to 0} 
  \begin{cases}
    -1 & \text{if}\ r>R(t),\\
    1& \text{if}\ r<R(t),
  \end{cases}
\end{equation}
where $R(t)= \sqrt[n]{r_0^n+nat}$ is the radius of the level set $\{c_\eps(x,t)= 0\} = \partial B_{R(t)}(0)$.

In order to determine $\tilde{p}_\eps$ we use that \eqref{eq:RotNSCH1} and \eqref{eq:ExplicitU} imply
\begin{eqnarray}\nonumber
  \partial_r \tilde{p}_\eps &=&  -\eps \tfrac{n-1}r |\partial_r \tilde{c}_\eps|^2 - \eps \partial_r |\partial_r \tilde{c}_\eps|^2 +\nu \tfrac1{r^{n-1}}\partial_r \left(r^{n-1}\partial_r u\right)- \rho u\partial_r u\\\label{eq:Pressure}
&=&   -\eps \tfrac{n-1}r |\partial_r \tilde{c}_\eps|^2 - \eps \partial_r |\partial_r \tilde{c}_\eps|^2 + \partial_r \left(\tfrac{a(n-1)}{2n+2} r^{-2n+2}- \tfrac{\nu a(n-1)}n r^{-n}\right). %%% Checken!!!
\end{eqnarray}
Now we decompose $\tilde{p}_\eps=p_{1,\eps}+p_{2,\eps}+p_3$ such that
\begin{equation*}
  \partial_r p_{1,\eps}(r) =  -\eps (n-1)\tfrac1r |\partial_r \tilde{c}_\eps(r)|^2,\qquad p_{2,\eps}(r)=- \eps  |\partial_r \tilde{c}_\eps(r)|^2\quad \text{for all}\ r>1.
\end{equation*}
Hence up to a constant
\begin{equation*}
  p_3= \frac{a(n-1)}{2n+2} r^{-2n+2}- \frac{\nu a(n-1)}n r^{-n}.
\end{equation*}
Because of the explicit form of $\tilde{p}_\eps$ and 
\begin{equation}\label{eq:DerivLimitSol}
  \partial_r \tilde{c}_\eps(r,t)= -\frac1\eps \theta'\left(\frac{\sqrt[n]{r^n-ant}-r_0}\eps\right) r^{n-1}(r^n-ant)^{\frac1n-1},
\end{equation}
it is easy to observe that
\begin{equation*}
  \tilde{p}_\eps(r)\to_{\eps\to 0} \tilde{p}_0 (r) \qquad \text{for all}\ r\neq R(t)
\end{equation*}
for some smooth $\tilde{p}_0\colon (1,M)\setminus \{R(t)\}\to \R$.
Now we consider
\begin{equation*}
  [p_{j,\eps}]_{R(t),\delta} := p_j(R(t)+\delta)-p_j(R(t)-\delta)= \int_{R(t)-\delta}^{R(t)+\delta} \partial_r p_j(s,t) \sd s,
\end{equation*}
which converges as $\eps\to 0$ and $\delta \to 0$ (in that order) to several contributions of $[\tilde{p}_0]$ at $R(t)$. 
For $j=2$ we have that
\begin{equation*}
  [p_{2,\eps}]_{R(t),\delta} = -\eps |\partial_r \tilde{c}_\eps(R(t)+\delta)|^2+ \eps |\partial_r \tilde{c}_\eps(R(t)-\delta)|^2 =0 
\end{equation*}
if $\eps <\delta$. Hence
\begin{equation*}
  \lim_{\eps\to 0} [p_{2,\eps}]_{R(t),\delta}=0.
\end{equation*}
Moreover, since $p_3$ is independent of $\eps$ and continuous, we have 
\begin{equation*}
  \lim_{\delta \to 0} \lim_{\eps\to 0}[p_3]_{R(t),\delta} =0.
\end{equation*}
Finally, using \eqref{eq:DerivLimitSol} we obtain
\begin{eqnarray}\nonumber
 [\tilde{p}_0]_{R(t),\delta}&=& \lim_{\eps\to 0} [\tilde{p}_\eps]_{R(t),\delta}=\lim_{\eps\to 0} [p_{1,\eps}]_{R(t),\delta}\\\nonumber
&=& \tfrac{n-1}\eps \int_{R(t)-\delta}^{R(t)+\delta}\left|\theta'\left(\frac{\sqrt[n]{r^n-ant}-r_0}\eps\right)\right|^2 r^{2n-3}(r^n-ant)^{\frac2n-2} \sd r\\\nonumber
&=&\left.\sigma (n-1) r^{2n-3}(r^n-ant)^{\frac2n-2}\right|_{r=R(t)}\\\label{eq:LimitSolution}
&=&\sigma\left(\frac{R(t)}{r_0}\right)^{2n-2}\frac{n-1}{R(t)},
\end{eqnarray}
where $\sigma:=\int_{\R}|\theta'(s)|^2 \sd s$. Here $R(t)>r_0$ for all $t>0$ and $R(t)\to_{t\to \infty} \infty$. The exact solution of the classical sharp interface model, i.e., \eqref{eq:1}-\eqref{eq:6} with $m_0=0$ and $\Omega^+_0= B_{r_0}(0)\setminus \ol{B_1(0)}$, $\ve_0= ar^{1-n}\tfrac{x}{|x|}$, is given by 
\begin{equation*}
  {\ve} (x,t)= ar^{1-n} \frac{x}{|x|},\quad
 \Omega^+(t) = B_{R(t)} (0)\setminus \ol{B_1(0)},
\end{equation*}
where $R(t)=\sqrt[n]{r_0^n+ant}$ as before, ${p}\colon \Omega\times (0,T)\to\R$ is constant in $\Omega^\pm(t)$ such that
\begin{equation}\label{eq:ExactSolution}
   [{p}](x,t) = \sigma\frac{n-1}{R(t)} \quad \text{on}\ \partial \Omega^+(t)=\Gamma(t).  %%% Checken!
\end{equation}
Hence the pressure $\tilde{p}$ of the limit solution as $\eps\to 0$ differs from the solution of the sharp interface \eqref{eq:ExactSolution} by a time dependent factor $\left(\frac{R(t)}{r_0}\right)^{2n-2}>1$, which corresponds to an increased surface tension coefficient that even increases strictly in time. 

\begin{rem}
  From the explicit solution \eqref{eq:LimitSolution2} one observes 
  \begin{equation*}
    \partial_r \tilde{c}_\eps(R(t),t)=-\frac1\eps \theta'(0) \left(\frac{R(t)}{r_0}\right)^{n-1}.
  \end{equation*}
Hence  $|\partial_r \tilde{c}|$ increases at the diffuse interface ``$r\approx R(t)$'' as $t$ increases, cf. Figure~\ref{fig:1}
\begin{figure}[h]
  \centering
  \includegraphics[width=9cm, height=6cm]{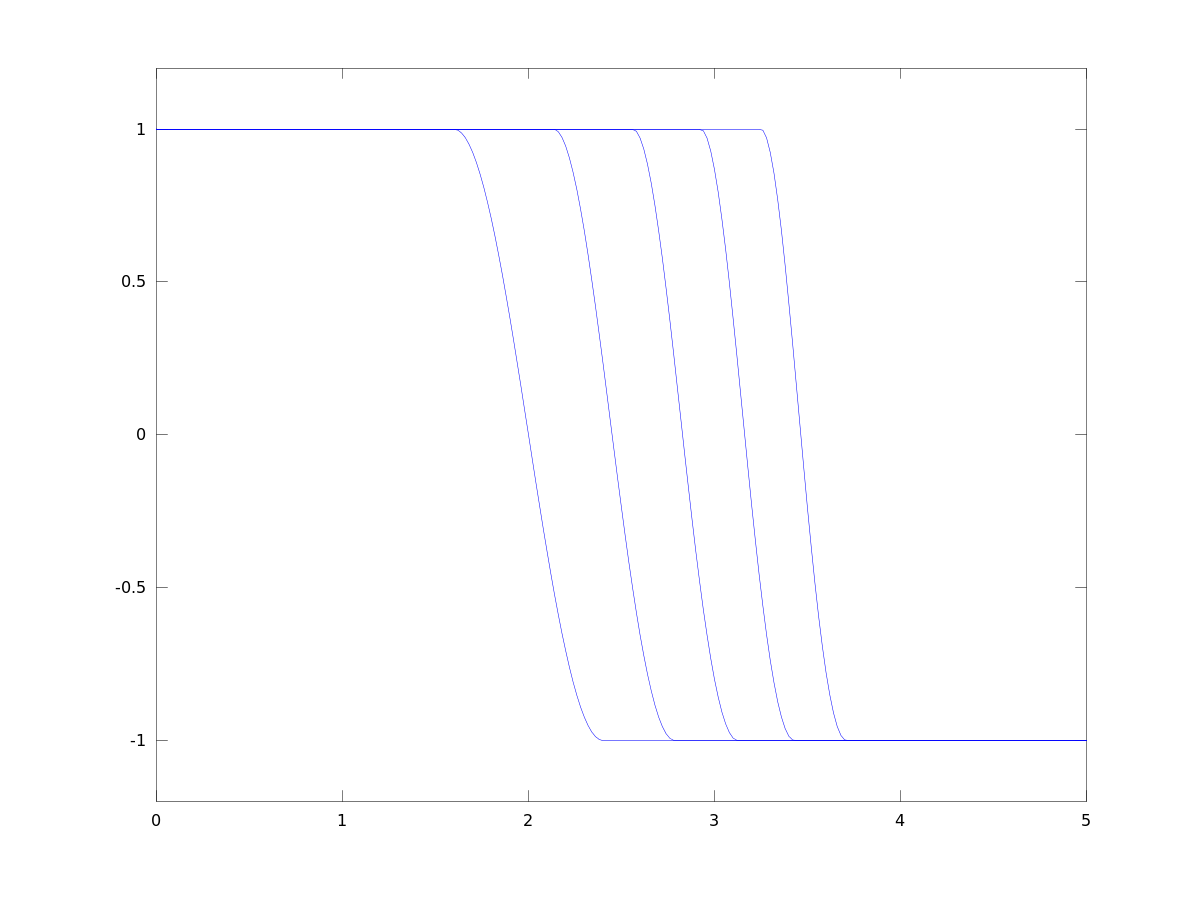}
  \caption{Plot of $\tilde{c}_\eps^\infty$ for $t=0,1,2,3,4$ (from left to right) with $a=1,\eps=0.4, r_0=2, n=2$.}
  \label{fig:1}
\end{figure}
\end{rem}
Finally, we determine the limit of the discrepancy measure:
\begin{eqnarray*}
  \int_\Omega\xi_\eps (c^\infty_\eps) \varphi\sd x   &=& \int_1^M \left( \eps\frac{|\partial_r \tilde{c}_\eps^\infty(r)|^2}2 - \frac{f(\tilde{c}_\eps^\infty(r))}\eps\right) \tilde{\varphi}(r)\, r^{n-1}d r \\
&=& \int_1^M \frac1\eps\left|\theta'\left(\frac{\sqrt[n]{r^n-ant}-r_0}\eps\right)\right|^2\left(\frac{r^{n-1}}{(r^n-ant)^{1-\frac1n}}\right) \tilde{\varphi}(r)\, r^{n-1}d r \\
&& -\int_1^M \frac1\eps f\left(\theta\left(\frac{\sqrt[n]{r^n-ant}-r_0}\eps\right)\right) \tilde{\varphi}(r)\, r^{n-1}d r \\
&\to_{\eps\to 0}&  \left(\sigma\kappa(t)-\tilde{\sigma}\right) \int_{\partial B_{R(t)}} \varphi(x)\sd x
\end{eqnarray*}
for all $\varphi \in C_0^\infty(\Omega)$ where  $\tilde{\sigma}:= \int_{\R} f(\theta(s))\sd s$, $\sigma= \int_{\R}|\theta'(s)|^2\sd s$ as before, and
\begin{equation*}
  \tilde{\varphi}(r)= \int_{\partial B_r(0)} \varphi(x)\sd x\qquad\text{for all}\ r\in (1,M).
\end{equation*}
Hence
\begin{equation}\label{eq:xiLimit}
  \xi_\eps (c^\infty_\eps)\to_{\eps\to 0} (\sigma\kappa(t)-\tilde{\sigma}) \delta_{\partial B_{R(t)}} \qquad \text{in}\ \mathcal{D}'(\Omega)
\end{equation}
since $\kappa(t)$ is strictly increasing in $t>0$, we have $\sigma\kappa(t)-\tilde{\sigma}\neq 0$ for all $t>0$  except possibly one.

\subsection{Nonconvergence in the Case $3<\alpha<\infty$} 
Based on the solution for the extreme case ``$\alpha=\infty$'' from the previous section, we will prove essentially the same result in the case $3<\alpha<\infty$. In order to avoid technical difficulties with the unboundedness  of $\{x\in \R^n:|x|>1\}$, we will consider \eqref{eq:RotNSCH1}-\eqref{eq:RotNSCH4} in 
\begin{equation*}
\Omega_M= \{x\in \Rn: 1<|x|<M\},  
\end{equation*}
 where $M>r_0>1$ is arbitrary, together with 
\begin{alignat}{3}\label{eq:RotBC1'}
u_\eps(1,t)&= a, &\qquad c_\eps (1,t)&= 1\qquad &\text{for all } t\in (0,T),  
\\
\label{eq:RotBC3}
u_\eps(M,t)&= \frac{a}{M^{n-1}}, &\qquad c_\eps (M,t)&= -1\qquad &\text{for all } t\in (0,T).  
\end{alignat}
\begin{defn}[{\bf Weak Solutions}]~\\
Let
\begin{eqnarray*}
  H^1_{(0)}&=& \left\{u\in C^0([1,M]): r^{(n-1)/2}\partial_r u\in L^2(1,M), \int_1^M u(r) r^{n-1}\sd r=0 \right\}%\\
%H^1_0 &=& \left\{u\in C^0([1,M]): u(0)=u(M)=0, r^{(n-1)/2}\partial_r u\in L^2(1,M) \right\}
\end{eqnarray*}
be equipped with the inner product
\begin{equation*}
  (u,v)_{H^1_{(0)}} := \int_1^\infty \partial_r u(r)\partial_rv(r) r^{n-1} \sd r
\end{equation*}
for all $u,v\in H^1_{(0)}$.% $H^1_0$, respectively.
 We embed $L^2_{(0)}(1,M)\hookrightarrow H^{-1}_{(0)}:= (H^1_{(0)})'$ by identifying $u\in L^2(1,M)$ with
\begin{equation*}
  \weight{u,\varphi}_{H^{-1}_{(0)},H^1_{(0)}}:= \int_1^M u(r) \varphi(r) r^{n-1}\sd r\qquad \text{for all}\ \varphi \in H^1_{(0)}. 
\end{equation*}
We call $(\tilde{c}_\eps,\tilde{\mu}_\eps)$ a weak solution of \eqref{eq:RotNSCH3}-\eqref{eq:RotNSCH4} together with \eqref{eq:RotBC1'},\eqref{eq:RotBC3} and $\tilde{c}_\eps|_{t=0}= \tilde{c}_{0,\eps}$ if
    \begin{alignat*}{1}
      \tilde{c}_\eps-\chi &\in C([0,T];H^1_0(1,M))\cap L^2(0,T;H^3(1,M)),\quad\\
 \partial_t \tilde{c}_\eps &\in L^2(0,T;H^{-1}_{(0)}(1,M)), \quad
\tilde{\mu}_\eps \in L^2(0,T;H^1(1,M)),
    \end{alignat*}
where $\chi \in C^\infty([1,M])$ with $\chi(1)=1, \chi(M)=-1$ and
\begin{equation*}
  \weight{\partial_t \tilde{c}_\eps(t),\varphi }_{H^{-1}_{(0)},H^1_{(0)}} + \int_1^M ar^{-n+1} \partial_r \tilde{c}_\eps(r,t) \varphi(r)\, r^{n-1} \sd r  = -m_0\eps^\alpha\int_1^M \partial_r \mu_\eps\partial_r \varphi \, r^{n-1}\sd r    
\end{equation*}
for almost every $t\in (0,T)$ and for all $\varphi \in H^1_{(0)}(1,M)$, \eqref{eq:RotNSCH4} is satisfied pointwise almost everywhere, and $c|_{t=0}= c_{0,\eps}$ in $H^1(1,M)$.
\end{defn}
Existence of weak  solutions can be proved by standard methods. E.g. it follows from \cite[Theorem 3.1]{AsymptoticCH} applied to $H_1= H^1_{(0)}$, $H_0= H^{-1}_{(0)}$,
\begin{eqnarray*}
  \varphi(u) &=& \int_1^M \left(\eps\frac{|\partial_r u(r)|^2}2 + \eps^{-1} f_0(u(r))\right)\, r^{n-1}\sd r \\
\weight{\mathcal{B}(v),w}_{H^{-1}_{(0)},H^1_{(0)}} &=& m_0\eps^{\alpha-1} \beta\int_1^M \partial_r v(r)\partial_rw(r) r^{n-1}\sd r -\int_1^Ma\partial_rv(r) w(r)\sd r    
\end{eqnarray*}
for all $v,w\in H_1$ and $u\in \operatorname{dom}(\varphi):= \{v\in H^1(1,M):\tfrac1{M-1} \int_1^M v(r) r^{n-1}\sd r =m\}$, where $m:= \int_1^M \tilde{c}_{0,\eps}(r)r^{n-1} dr$, $f_0(s):= f(s) -\tfrac{\beta}2 s^2$ for all $s\in\R$ and $\beta:= \inf_{s\in\R} f''(s)$. Then $f_0$ and $\varphi$ are convex and the subgradient $\mathcal{A}=\partial\varphi$ taken with respect to $H^{-1}_{(0)}$ satisfies
\begin{eqnarray*}
\lefteqn{\weight{\mathcal{A}(u),w}_{H^{-1}_{(0)},H^1_{(0)}}}\\
 &=& m_0\eps^{\alpha+1}\int_1^M \partial_r(r^{-n+1}\partial_r (r^{n-1}\partial_r u))\partial_rw\, r^{n-1}\sd r+ m_0\eps^{\alpha-1}\int_1^M f_0'(u(r))  r^{n-1}\sd r    
\end{eqnarray*}
for all $u\in \mathcal{D}(\partial \varphi)= \{v\in H^3(1,M): \tfrac1{M-1}\int_1^M v(r) r^{n-1}\sd r =m\}$. Then it is easy to verify that all conditions of \cite[Theorem 3.1]{AsymptoticCH} are satisfied.

Finally, if $(\tilde{c}_\eps,\mu_\eps)$ is a weak solution as above, we can choose $\varphi = \mu_\eps -\bar{\mu}_\eps$ with $\bar{\mu}_\eps= \int_1^M \mu_\eps (r)r^{n-1}\sd r$ in the weak formulation of the convective Cahn-Hilliard equation and obtain the energy identity
\begin{eqnarray}\label{eq:EnergyId}
 \lefteqn{ \int_1^M \left(\eps|\partial_r \tilde{c}_\eps(r,t)|^2 + \eps^{-1} f(\tilde{c}_\eps(r,t))\right)\, r^{n-1}\sd r}\\\nonumber
&& + \int_0^t \int_1^M m_0\eps^\alpha|\nabla\mu_\eps(r,\tau)|^2 r^{n-1}\sd r \sd \tau
=   \int_1^M \left(\eps|\partial_r \tilde{c}_{0,\eps}(r)|^2 + \eps^{-1} f(\tilde{c}_{0,\eps}(r))\right)\, r^{n-1}\sd r
\end{eqnarray}
for all $t\in (0,T)$.
  \begin{thm}
    Let $\kappa>3$, $r_0\in (1,M)$, $0<\delta<\min (r_0-1,M-r_0)$ and $T>0$ such that $R(T)=\sqrt[n]{r_0^n+na T}<M-\delta$, $\Omega= \{x\in\R^n: 1<|x|<M\}$, $\tilde{c}_{0,\eps}$ and $\theta$ be as in \eqref{eq:RotIV}-\eqref{eq:theta}, and let $(\ve_\eps,p_\eps,c_\eps,\mu_\eps)$ be the radially symmetric solutions of the form \eqref{eq:RadiallySolutions} 
of \eqref{eq:NSCH1}-\eqref{eq:NSCH5}, \eqref{eq:NSCH7} and boundary conditions \eqref{eq:RotBC1}-\eqref{eq:RotBC2}, $\no\cdot \nabla \mu_\eps|_{\partial\Omega}=0$. Then 
\begin{equation*}
\ve_\eps \equiv ar^{-n+1}\vc{e}_r %,  \qquad \vc{e}_r=\frac{x}{|x|},
\end{equation*}
 and 
    \begin{alignat*}{2}
      c_\eps&\to_{\eps\to 0} 2\chi_{B_{R(t)}(0)}-1 &\qquad&\text{for every}\ x\in \Omega\setminus\partial B_{R(t)}(0),t\in (0,T),\\
      p_\eps&\to_{\eps\to 0} p &\qquad& \text{in}\ \mathcal{D}'(\Omega\times(0,T)),
    \end{alignat*}
    where $R(t)= \sqrt[n]{r_0^n+nat}$, $p\in \mathcal{D}'(\Omega\times (0,T))$ coincides with a function that is continuous in $x\in\Omega\setminus \partial B_{R(t)}(0)$ for every $t\in (0,T)$, and
    \begin{equation*}
      [p] = {\kappa(t)\sigma H}\qquad \text{on}\ \Gamma(t)= \partial B_{R(t)}(0)\ \text{for all}\ t\in (0,T)  
    \end{equation*}
    and {$1<\kappa(t)=\left(\frac{R(t)}{r_0}\right)^{2n-2}\to_{t\to\infty} \infty$}. Moreover, 
    \begin{equation}\label{eq:xiConv}
      \xi_\eps(c_\eps)= \frac{\eps|\nabla c_\eps|^2}2 - \frac{f(c_\eps)}\eps \to_{\eps\to 0}(\sigma\kappa(t)-\tilde{\sigma}) \delta_{\partial B_{R(t)}} \qquad \text{in}\ \mathcal{D}'(\Omega\times (0,T)), 
    \end{equation}
where $\sigma= \int_{\R} |\theta'(s)|^2\sd s, \tilde{\sigma}= \int_{\R} f(\theta(s))\sd s$.
  \end{thm}
  \begin{proof}
First of all, we show that $\|\tilde{c}_\eps\|_{L^\infty((1,M)\times (0,T))}$ is uniformly bounded. To this end let $W(c)$ be as in Section~\ref{sec:SharpInterfaceLimit}. Then as in \eqref{eq:NablaWepsEstim}
\begin{eqnarray*}
  \int_1^M|\partial_r W(\tilde{c}_{\eps}(r,t))| r^{n-1}\sd r &\leq& C \int_1^M\left(\eps \frac{|\partial_r \tilde{c}_{0,\eps}(r)|^2}2 + \frac{f(\tilde{c}_{0,\eps}(r))}\eps\right) r^{n-1}\sd r\leq C'
\end{eqnarray*}
by \eqref{eq:EnergyId} and the choice of the initial data. 
Hence
\begin{equation*}
  \sup_{0<t<T, 0<\eps <1}\|W(\tilde{c}_\eps(t))\|_{L^\infty(1,M)}\leq \sup_{0<t<T,0<\eps<1}\|\partial_r W(\tilde{c}_\eps(t))\|_{L^1(1,M)}\leq C'
\end{equation*}
due to $\tilde{c}_\eps (t,M)=1$,
which implies
\begin{equation}\label{eq:LinftyEstim}
\sup_{0<\eps<1}  \|\tilde{c}_\eps\|_{L^\infty((1,M)\times (0,T))} \leq \tilde{M}
\end{equation}
for some $\tilde{M}>0$
due to \eqref{eq:WEstim}.
    Now let $d_\eps := \tilde{c}_\eps-\tilde{c}_\eps^\infty$, where $\tilde{c}_\eps^\infty$ is as in \eqref{eq:LimitSolution2}.
First we will show 
\begin{equation*}
  \int_0^T\eps\|\partial_r d_\eps(t)\|_{L^2}^2 \sd t \to_{\eps\to 0}0,
\end{equation*}
where $L^2= L^2((1,M);r^{n-1}\sd r)$.
To this end we use that
\begin{equation}\label{eq:DiffEq}
  \weight{\partial_t d_\eps(t),\varphi} + m_0\eps^{\alpha+1}\int_1^M \tilde{\Delta} d_\eps(r,t) \tilde{\Delta}\varphi(r)\, r^{n-1}\sd r = \int_1^M g_\eps(r,t) \tilde{\Delta} \varphi(r)\, r^{n-1}\sd r
\end{equation}
for all $\varphi \in H^1_{(0)}$ and almost every $t\in (0,T)$,
where 
\begin{equation*}
\tilde{\Delta} u(r) = r^{-n+1}\partial_r (r^{n-1}\partial_r u(r))\qquad \text{for all}\ u\in H^2(1,M) 
\end{equation*}
and 
\begin{equation*}
  g_\eps(r,t) = m_0\eps^{\alpha+1}\tilde{\Delta} c_\eps^\infty(r,t)+m_0\eps^{\alpha-1}f'(\tilde{c}_\eps(r,t)). 
\end{equation*}
Moreover,
\begin{equation*}
  \|g_\eps(t)\|_{L^2}\leq m_0\eps^{\alpha+1} \|\tilde{\Delta}\tilde{c}_\eps^\infty(t)\|_{L^2} +m_0\eps^{\alpha-1}\|f'(\tilde{c}_\eps(t))\|_{L^2}\leq C(T)m_0\eps^{\alpha-\frac12}
\end{equation*}
where $L^2=L^2(1,M;r^{n-1}dr)$ and we have used that
\begin{equation*}
  \|\tilde{c}_\eps^\infty(t)\|_{H^2(1,M)}\leq C(T)\eps^{-\frac32} 
\end{equation*}
and 
\begin{equation*}
  \eps^{-1}\|f'(c_\eps(t)\|_{L^2}^2 \leq C E_\eps(c_\eps(t))\leq C'
\end{equation*}
due to $|f'(s)|^2\leq C(\tilde{M})f(s)$ for all $s\in [-\tilde{M},\tilde{M}]$ and \eqref{eq:LinftyEstim}.
Hence, choosing $\varphi=d_\eps(t)$ in \eqref{eq:DiffEq} and integrating in time, we conclude
\begin{eqnarray}\nonumber
  \lefteqn{\sup_{0\leq t\leq T} \|d_\eps(t)\|_{L^2}^2 + m_0\eps^{\alpha+1}\int_0^{T}\|\tilde{\Delta} d_\eps(t)\|_{L^2}^2 \sd t}\\\label{eq:DiffConv}
 &\leq& \int_0^{T} \eps^{-\frac{\alpha+1}2}\|g_\eps(t)\|_{L^2}\eps^{\frac{\alpha+1}2}\|\tilde{\Delta} d_\eps(t)\|_{L^2}\sd t.
\end{eqnarray}
Using the Cauchy-Schwarz and Young's inequality, we obtain
\begin{equation*}
{  \sup_{0\leq t\leq T} \|d_\eps(t)\|_{L^2}^2 + m_0\eps^{\alpha+1}\int_0^{T}\|\tilde{\Delta} d_\eps(t)\|_{L^2}^2 \sd t}
\leq C\int_0^{T} \eps^{-\alpha-1}\|g_\eps(t)\|_{L^2}^2\sd t\leq C(T) \eps^{\alpha-2} .
\end{equation*}
Combining this estimate with
\begin{equation*}
  \|\partial_r v\|_{L^2}^2 \leq C\|v\|_{L^2}\|v\|_{H^2}\leq C'\|v\|_{L^2}\|\tilde{\Delta}v\|_{L^2}\quad \text{for all}\ v\in H^2(1,M)\cap H^1_0(1,M), 
\end{equation*}
we conclude
\begin{equation}\label{eq:dEpsConv}
  \int_0^T\eps\|\partial_r d_\eps(t)\|_{L^2}^2 \sd t \leq C \eps^{\frac{\alpha-1}2}\eps^{-1}= C\eps^{\frac{\alpha-3}2}\to_{\eps\to 0}0
\end{equation}
since $\alpha>3$. 
% Furthermore, using
% \begin{eqnarray*}
%   \sup_{r\in (1,M)}|d_\eps(r,t)|^2 &=& |d_\eps(1,t)|^2 + \sup_{r\in (1,M)}2\int_1^r\partial_r d_\eps(s,t)d_\eps(s) \sd s\\
% &\leq& C(M) \|\partial_r d_\eps(t)\|_{L^2}\|d_\eps(t)\|_{L^2}
% \leq  C(M)\eps^{\alpha/2-1}\eps^{\alpha/2-2}\\
% &=& C(M) \eps^{\alpha-3}\leq C(M)
% \end{eqnarray*}
% uniformly in $t\in (0,T)$, $\eps\in (0,1)$, we conclude 
% \begin{equation}\label{eq:SupBdd}
%   \sup_{t\in (0,T),r\in (1,M)} |c_\eps(r,t)|\leq C(M) + \sup_{t\in (0,T),r\in (1,M)} |c_\eps^\infty(r,t)| =C(M)+1
% \end{equation}

In order to determine $\tilde{p}_\eps$ we use again \eqref{eq:Pressure}
% \begin{eqnarray*}
%   \partial_r \tilde{p}_\eps &=&  -\eps (n-1)\tfrac1r |\partial_r \tilde{c}|^2 - \eps \partial_r |\partial_r c|^2 +\nu \tfrac1{r^{n-1}}\partial_r \left(r^{n-1}\partial_r u\right)- \left(\rho u+\frac{\partial\rho}{\partial c}  m_0\eps^\alpha\partial_r \tilde{\mu}\right)\partial_r u\\
% &=&   -\eps (n-1)\tfrac1r |\partial_r c|^2e_r + \partial_r \left(\frac{a(n-1)}{2n+2} r^{-2n+2}- \frac{\nu a(n-1)}n r^{-n}\right)+\frac{\partial\rho}{\partial c}  m_0\eps^\alpha\partial_r \tilde{\mu}\partial_r u %%% Checken!!!
% \end{eqnarray*}
and decompose $\tilde{p}_\eps=p_{1,\eps}+p_{2,\eps}+p_3$ similarly as before, where
\begin{equation*}
  \partial_r p_{1,\eps} =  -\eps (n-1)\tfrac1r |\partial_r \tilde{c}_\eps|^2,\qquad p_{2,\eps}=- \eps \partial_r |\partial_r \tilde{c}_\eps|^2.
\end{equation*}
Hence up to a constant
\begin{equation*}
  p_3= \frac{a(n-1)}{2n+2} r^{-2n+2}- \frac{\nu a(n-1)}n r^{-n}
\end{equation*}
as before.
Moreover, let
\begin{equation*}
  \partial_r p_{1,\eps}^\infty =  -\eps (n-1)\tfrac1r |\partial_r \tilde{c}_\eps^\infty|^2,\qquad p_{2,\eps}^\infty=- \eps \partial_r |\partial_r \tilde{c}_\eps^\infty|^2
\end{equation*}
be the corresponding parts of the pressure for the case $\alpha=\infty$. Then \eqref{eq:dEpsConv} implies that
\begin{alignat*}{2}
  \partial_rp_{1,\eps} - \partial_rp_{1,\eps}^\infty &\to_{\eps\to 0} 0 &\quad& \text{in}\ L^1((1,M)\times (0,T)),\\
  \partial_rp_{2,\eps} - \partial_rp_{2,\eps}^\infty &\to_{\eps\to 0} 0 &\qquad& \text{in}\ \mathcal{D}'((1,M)\times (0,T)).
\end{alignat*}
% Since $m_0\eps^{\alpha/2}\|\partial_r\mu_\eps\|_{L^2}\leq C$ by the energy estimate, we conclude
% \begin{equation*}
%   \|p_4\|_{L^\infty}\leq Cm_0\eps^{\alpha}\|\partial_r\mu_\eps\|_{L^2}\leq C'\eps^{\alpha/2}\to_{\eps\to 0}0.
% \end{equation*}
Since $p_3$ is independent of $\eps$ and the same as in the case $\alpha=\infty$, we conclude that
\begin{equation*}
  \tilde{p}_0:=\lim_{\eps\to 0} \tilde{p}_\eps  =  \lim_{\eps\to 0} \tilde{p}_\eps^\infty = p_0^\infty \qquad \text{in}\ \mathcal{D}'((1,M)\times (0,T)),
\end{equation*}
where $\tilde{p}_\eps^\infty$ is the pressure in the case $\alpha =\infty$ and $\tilde{p}_0^\infty$ is its limit as $\eps\to 0$.
Therefore 
\begin{eqnarray*}
 [\tilde{p}_0]&=&\sigma\left(\frac{R(t)}{r_0}\right)^{2n-2}\frac{n-1}{R(t)},
\end{eqnarray*}
where $\sigma:=\int_{\R}|\theta'(s)|^2 \sd s$, by the result of the previous section.

Finally, it remains to prove \eqref{eq:xiConv}. First of all, because of \eqref{eq:dEpsConv}, we conclude
\begin{equation*}
  \lim_{\eps\to 0}\left(\eps\frac{|\nabla c_\eps^\infty|^2}2- \eps\frac{|\nabla c_\eps|^2}2 \right) = 0 \qquad \text{in}\ L^1(\Omega\times(0,T)),
\end{equation*}
where $c_\eps^\infty(x,t)= \tilde{c}_\eps^\infty(|x|,t)$.
Moreover, using \eqref{eq:LinftyEstim} and a Taylor expansion of $f(\tilde{c}_\eps(r,t)))$ around $\tilde{c}_\eps^\infty(r,t)$, we conclude
\begin{eqnarray*}
  \lefteqn{\int_0^T\int_1^M\left|\frac{f(\tilde{c}_\eps^\infty(r,t))}\eps-\frac{f(\tilde{c}_\eps(r,t))}\eps\right|\sd r\sd t}\\
&\leq & \int_0^T\int_1^M\left|\frac{f'(\tilde{c}_\eps^\infty(r,t))d_\eps(r,t)}\eps\right|\sd r\sd t +C\int_0^T\int_1^M\left|\frac{d_\eps(r,t)^2}{\eps}\right|\sd r\sd t \\
&\leq & C(M,T)\eps^{-1}\left( \|d_\eps\|_{L^2(\Omega\times(0,T))}^2+ \|d_\eps\|_{L^2(\Omega\times(0,T))}\|f'(\tilde{c}_\eps^\infty)\|_{L^2(\Omega\times (0,T))}\right)\\
&\leq& C'(M,T) \left(\eps^{\alpha-3} + \eps^{\frac{\alpha-3}2}\right)\to_{\eps\to 0} 0
\end{eqnarray*}
since $\|\tilde{c}_\eps\|_{L^\infty(\Omega\times (0,T))}$ and $\eps^{\frac12}\|f'(\tilde{c}_\eps^\infty)\|_{L^2(\Omega\times (0,T))}$ are uniformly bounded in $\eps\in (0,1)$ due to $|f'(c_\eps^\infty)|^2\leq Cf(c_\eps^\infty)$.
Altogether we obtain
\begin{equation*}
  \lim_{\eps\to 0} \left(\xi_\eps(c_\eps^\infty)-\xi_\eps(c_\eps) \right) =0 \qquad \text{in}\ L^1(\Omega\times (0,T)),
\end{equation*}
which implies \eqref{eq:xiConv} due to \eqref{eq:xiLimit}.
  \end{proof}
%%% Local Variables:
%%% mode: latex
%%% TeX-master: "ChenSharpInterfaceLimit.tex"
%%% End:

\medskip

\noindent
{\bf Acknowledgments:} The authors acknowledge support from the German
Science Foundation through Grant Nos. AB285/4-1.

%\bibliography{Bibliography}
%\bibliographystyle{plain}

\def\cprime{$'$} \def\ocirc#1{\ifmmode\setbox0=\hbox{$#1$}\dimen0=\ht0
  \advance\dimen0 by1pt\rlap{\hbox to\wd0{\hss\raise\dimen0
  \hbox{\hskip.2em$\scriptscriptstyle\circ$}\hss}}#1\else {\accent"17 #1}\fi}

\end{document}